\documentclass[twoside]{amsart}

\newcommand{\sbc}{\on{sbc}}
\newcommand{\bc}{\on{bc}}

\newcommand{\Asai}{\on{Asai}}
\newcommand{\Mat}{\on{Mat}}
\newcommand{\Rog}{\on{Rog}}
\newcommand{\AI}{\on{AI}_{E/F}}

\newcommand{\on}{\operatorname}
\newcommand{\ol}{\overline}

\renewcommand{\Re}{\on{Re}}

\newcommand{\gm}{\gamma}

\newcommand{\quo}[1]{#1(F)\bs #1(\A)}
\newcommand{\il}{\int\limits_}
\newcommand{\iq}[1]{\il{\quo{#1}}}

\newcommand{\f}{\mathfrak}
\renewcommand{\c}{\mathcal}
\newcommand{\into}{\hookrightarrow}

\newcommand{\Fr}{\operatorname{Fr}}

\newcommand{\C}{\mathbb{C}}

\newcommand{\A}{\mathbb{A}}
\newcommand{\Q}{\mathbb{Q}}
\newcommand{\Z}{\mathbb{Z}}
\newcommand{\R}{\mathbb{R}}

\newcommand{\bpm}{\begin{pmatrix}}
\newcommand{\epm}{\end{pmatrix}}
\newcommand{\bsm}{\begin{smallmatrix}}
\newcommand{\esm}{\end{smallmatrix}}
\newcommand{\bspm}{\left(\begin{smallmatrix}}
\newcommand{\espm}{\end{smallmatrix}\right)}
\newcommand{\bbm}{\begin{bmatrix}}
\newcommand{\ebm}{\end{bmatrix}}

\usepackage{mathrsfs}
\usepackage{amsfonts}
\usepackage{amssymb}
\usepackage{amssymb}
\usepackage{amssymb}
\usepackage{amssymb}
\usepackage{amssymb}
\usepackage{amssymb}
\usepackage{amssymb}
\usepackage{amssymb}
\usepackage{amsmath}
\usepackage{amsthm}
\usepackage[all]{xy}

\usepackage{lastpage}

\RequirePackage{amsmath} \RequirePackage{amssymb}
\usepackage{amscd,latexsym,amsthm,amsfonts,amssymb,amsmath,amsxtra}
\usepackage
[
colorlinks=true, 
linkcolor=blue,
urlcolor=blue,
bookmarks=true,
bookmarksopen=true,
citecolor=blue,
]
{hyperref}

    \newcommand{\BA}{{\mathbb {A}}} 
    \newcommand{\BC}{{\mathbb {C}}} 
     
    \newcommand{\BG}{{\mathbb {G}}}

     \newcommand{\BR}{{\mathbb {R}}}

    \newcommand{\CS}{{\mathcal {S}}} 
     \newcommand{\CV}{{\mathcal {V}}}
    \newcommand{\CW}{{\mathcal {W}}}

     \newcommand{\fo}{{\mathfrak{o}}}  \newcommand{\fp}{{\mathfrak{p}}}

     \newcommand{\fu}{{\mathfrak{u}}}

       \newcommand{\fP}{{\mathfrak{P}}}

     \newcommand{\RP}{{\mathrm {P}}}
     
    \newcommand{\RS}{{\mathrm {S}}} 
    \newcommand{\RU}{{\mathrm {U}}}

    \newcommand{\lenth}{{\mathrm {\lenth}}}

    \newcommand{\Ad}{{\mathrm{Ad}}}

    \newcommand{\Gal}{{\mathrm{Gal}}} \newcommand{\GL}{{\mathrm{GL}}}
    \newcommand{\Hom}{{\mathrm{Hom}}} 
    \newcommand{\Ind}{{\mathrm{Ind}}} \newcommand{\ind}{{\mathrm{ind}}}

    \newcommand{\ord}{{\mathrm{ord}}}

    \renewcommand{\Re}{{\mathrm{Re}}} 
    \newcommand{\Res}{{\mathrm{Res}}}

 \newcommand{\Vol}{{\mathrm{Vol}}}

\newcommand{\supp}{\mathrm{supp}}

 \newcommand{\SL}{{\mathrm{SL}}}
 \newcommand{\SO}{{\mathrm{SO}}}
\newcommand{\SU}{{\mathrm{SU}}}

\newcommand{\vol}{{\mathrm{vol}}}

\newcommand{\diag}{{\mathrm{diag}}}\newcommand{\ch}{{\mathrm{ch}}}

    \newcommand{\bx}{{\bf {x}}}   \newcommand{\bm}{{\bf {m}}}

    \newcommand{\wt}{\widetilde} \newcommand{\wh}{\widehat}

    \newcommand{\wpair}[1]{\left\{{#1}\right\}}

    \newcommand{\incl}{\hookrightarrow}
    
     \newcommand{\ra}{\rightarrow}

    \newcommand{\bs}{\backslash}
    
    \theoremstyle{plain}
    
       \newtheorem*{theorem*}{Theorem}

    \newtheorem{thm}{Theorem}[section] \newtheorem{cor}[thm]{Corollary}
    \newtheorem{lem}[thm]{Lemma}  \newtheorem{prop}[thm]{Proposition}
    \newtheorem {conj}[thm]{Conjecture} \newtheorem{defn}[thm]{Definition}
    \newtheorem {rmk}[thm]{Remark}
    
    \numberwithin{equation}{section}

\newcommand{\ssm}{\smallsetminus}
\usepackage[top=1in, bottom=1in, left=1.25in, right=1.25in]{geometry}

\title{Adjoint $L$-functions for $\GL(3)$ and $\RU_{2,1}$}
\author{Joseph Hundley}
\address{Department of Mathematics, State University of New York at Buffalo, Buffalo, NY, USA}
\email{ jahundle@buffalo.edu}
\author{Qing Zhang}
\address{Department of Mathematics and Statistics, University of Calgary, Canada}
\email{qing.zhang1@ucalgay.ca}

\begin{document}

\maketitle

\begin{abstract}
We show that the finite part of the adjoint $L$-function (including contributions from all non-archimedean places, including ramified places) is holomorphic in $\Re(s) \ge 1/2$ for a cuspidal automorphic representation of $\GL_3$ over a number field. This improves the main result of \cite{H18}. We obtain more general results for twisted adjoint $L$-functions of both $\GL_3$ and quasisplit unitary groups. For unitary groups, we explicate the relationship between poles of twisted adjoint $L$-functions, endoscopy, and the structure of the stable base change lifting.
\end{abstract}

\tableofcontents


\section{Introduction}
\subsection{Motivation and Background} Let $F$ be a number field and $\BA$ be its ring of adeles. Let $H$ be a connected, reductive $F$-group, isomorphic to $\GL_3$ over the algebraic closure $\ol F$ of $F.$ The $L$-group $^LH$ of $H$ is then the semi-direct product of $\GL_3(\C)$ with a Galois or Weil group, depending on which form of the $L$-group we consider (cf.\cite{BorelCorvallis}). Differentiating the action of $^LH$ on its own identity component by conjugation, we obtain an action 
on $\f{gl}_3(\C).$ The subspace $\f{sl}_3(\C)$ is preserved. We denote the 
action of $^LH$ on $\f{sl}_3(\C)$ by 
$\Ad.$

In this paper we study the twisted adjoint $L$-function $L(s, \pi, \Ad \times \chi)$ where $\chi$ is a Hecke character. This work is motivated by the following simple conjecture regarding the untwisted $L$-function in the split case. 
\begin{conj}\label{conj: adjoint GL3 is entire}
Let $\pi$ be an irreducible cuspidal automorphic representation of $\GL_n(\A)$, where $\A$ is the adele ring 
of a global field. Let $\Ad$ denote the adjoint representation of $\GL_n(\BC)$ on $\f{sl}_n(\C).$ Then the 
global Langlands $L$-function $L(s, \pi, \Ad)$ is entire. 
\end{conj}
In the case $n=2,$ Conjecture \ref{conj: adjoint GL3 is entire} was proved by Gelbart-Jacquet \cite{GelbartJacquet}, generalizing a method of Shimura \cite{Shimura}. A different proof was given in \cite{JZ}. Flicker \cite{Flicker} proved that under a certain assumption about ratios of Hecke $L$ series (see ``$(\on{Ass}; E, \omega),$'' \cite[p. 233]{Flicker}), Conjecture \ref{conj: adjoint GL3 is entire} holds if $\pi\cong \otimes_v \pi_v$ with at least one $\pi_v$ supercuspidal. More precisely, he proved that for such $\pi,$ the twisted adjoint $L$-function $L(s, \pi, \Ad \times \chi)$ is entire unless $\chi$ is nontrivial and $\pi \cong \pi \otimes \chi.$ This proof was based on a trace formula. Our approach is based on an integral representation, which was pioneered in \cite{G}, and a method of ruling out poles which was pioneered in \cite{GJ}.

Note that the action of $\GL_n$ on the space of $n\times n$ matrices may be regarded as the tensor 
product of the standard representation with its own dual. It decomposes as the direct sum of our 
adjoint representation, and the one dimensional span of the identity matrix equipped with trivial 
action. It follows that 
\begin{equation} \label{Ad = RS/zeta}
L( s, \pi, \Ad\times \chi) = \frac{L( s, \pi \times \wt \pi\times \chi )}{L(s, \chi)},
\end{equation}
where $\wt \pi$ is the contragredient of $\pi.$
From this perspective, Conjecture \ref{conj: adjoint GL3 is entire} may be 
viewed as saying that 
$L(s, \pi \times \wt \pi\times \chi)$ should be ``evenly divisible'' by $L(s, \chi).$
More concretely, it says that $L(s, \pi \times \wt \pi)$ must vanish at each zero 
of $\zeta(s),$ to at least the same order. 
(Cf. \cite[p. 234]{Flicker}.)

As explained in \cite{JngR}, Conjecture \ref{conj: adjoint GL3 is entire} is expected to be related to the conjecture that $\zeta_K(s)/\zeta_F(s)$ is entire for a field extension $K/F$. As explained in the introduction of \cite{BumpGinzburg}, since $\zeta_K(s)/\zeta_F(s)$ is a product of Artin $L$-functions, the latter conjecture is a consequence of the Artin conjecture.
This relationship is one display in 
Flicker's conditional result. For more details about this relationship, see \cite{JngR}.

In the case when the global field is $\Q,$ Conjecture \ref{conj: adjoint GL3 is entire} may also be viewed from the point of view of the Selberg class. We recall that the Selberg class is a class of meromorphic functions $\C \to \C$ introduced by Selberg in \cite{Selberg}. For $\pi$ a cuspidal automorphic representation of $\GL_n(\A),$ the finite $L$-function $L_f (s, \pi \times \wt \pi)$ will be an element of the Selberg class, unless $\pi$ is a counterexample to the Ramanujan conjecture. It then would follow from the conjectures in \cite{Selberg} that $L_f(s, \pi \times \wt \pi)$ must be the product of the Riemann zeta function and another element of the Selberg class (cf. remark ii at the bottom of p. 370 of \cite{Selberg}). In other words it would follow that $L_f(s, \pi, \Ad)$ is itself an element of the Selberg class. But then it must be entire since elements of the Selberg class have no poles except possibly at one. 

\subsection{Main results}
For $n>2$ the adjoint $L$-function of $\GL_n$ is  not accessible via the Langlands-Shahidi method. Some information can be obtained through the relationship \eqref{Ad = RS/zeta} with the better-understood Rankin-Selberg  $L$-function. In particular, one can get a global functional equation of the adjoint $L$-function $L(s,\pi,\Ad\times \chi)$ by the global functional equations of $L(s,\pi\times \wt \pi\times \chi)$ and $L(s,\chi)$. To obtain further information of the adjoint $L$-function in the case $n=3$, our main tool here is an integral representation pioneered in \cite{G} together with a method of ruling out poles pioneered in \cite{GJ}.

Ginzburg's local zeta integral is based on an embedding of $\SL_3$ into a split exceptional group of type $G_2.$ The argument was adapted in \cite{H12} to apply to special quasisplit unitary group $\SU_{2,1}$ (which depends on a fixed quadratic extension of $F$), which also embed into $G_2.$ Let $H$ be one of these groups, regarded as a subgroup of $G_2.$ Then Ginzburg's global zeta integral is given by 
$$Z(\varphi, f_s)=\int_{H(F) \bs H(\A)} \varphi(g) E(g,f_s) \, dg,$$
where $\varphi$ is a cuspidal automorphic form on $\GL_3(\BA)$ or $\RU_{2,1}(\BA)$ in a fixed irreducible cuspidal automorphic representation $\pi$  and $E(g, f_s)$ is an Eisenstein series on $G_2$. See \cite{G, H18} or $\S$\ref{sec: split main theorem} for more details. In the split case (when $H=\SL_3$), among other things, Ginzburg proved that the above local zeta integral is Eulerian
$$Z(\varphi,f_s)=\prod_v Z^*(W_v,f_{s,v}),$$
with
$$Z^*(W_v,f_{s,v})=L(3s,\chi_v)L(6s-2,\chi_v^2)L(9s-3,\chi_v^3)\int_{N_2(F_v)\setminus H(F_v)}W_v(g)f_s(w_\beta g)dg,$$
where $\chi=\otimes_v\chi_v$ is a character of $F^\times\setminus \BA^\times$ which is part of the datum of the Eisenstein series, $W_v$ is the Whittaker function associated with local component of $\varphi$,  $N_2$ is certain subgroup of $H=\SL_3$, and $w_\beta$ is the Weyl element associated with the long simple root $\beta$ of $G_2$.  Ginzburg also showed that at unramified places, the local zeta integral $Z^*(W_v,f_{s,v})$ represents the adjoint $L$-function $L(s,\pi_v,\Ad\times \chi_v)$.

It is worth noting that both $Z(\varphi, f_s)$ depends only on the restriction of $\varphi$ to $\SL_3,$ and hence that one could, in principle begin with an automorphic representation of $\SL_3(\A)$ instead of $\GL_3(\A).$ However, the unfolding argument in \cite{G} results in a particular Whittaker integral of $\varphi.$ Thus, one must restrict attention to representations which are globally generic with respect to this particular character. Now, each cuspidal automorphic representation of $\SL_3(\A)$ may be obtained as one of the components in the restriction of some cuspidal automorphic representation of $\GL_3(\A)$ (unique up to twist). As explained in \cite[p.120]{BumpGinzburg}, the phenomenon that the integral $Z(\varphi,f_s)$ only depends on its restriction to $\SL_3(\A)$ is a reflection of the fact that the the adjoint representation of $\GL_3(\BC)$ factors through $\RP\GL_3(\BC)$, the $L$-group of $\SL_3$. Similarly, the local zeta integral $Z^*(W_v,f_{s,v})$ also only depends on $ W_v|_{\SL_3(F_v)}$.

To obtain the holomorphy of $L(s,\pi,\Ad\times \chi)$, on one hand, one needs to analyze the properties of the global integral $Z(\varphi,f_s)$; on the other hand, one needs to analyze the local zeta integral $Z^*(W_v,f_{s,v})$ at every ramified place $v$. 
Because of the functional equation, 
it suffices to rule out poles in 
the half plane $\Re(s) \ge \frac 12.$
After the pioneering work of Ginzburg \cite{G} and Ginzburg-Jiang \cite{GJ}, some progress in this direction was obtained in \cite{H18}. The goal of our first main result in this paper is to extend \cite[Theorem 6.1]{H18}, which treats the partial $L$-function $L^S(s,\pi,\Ad\times \chi)=\prod_{v\notin S}L(s,\pi_v,\Ad\times \chi_v)$ attached to a finite set $S=S(\pi,\chi)$ of places of $F$. Here $S(\pi,\chi)$ contains the set of all infinite places $S_\infty$ and all of the ramified places of $\pi$ and $\chi$. The main result in \cite{H18} states that for $\chi =1$ or $\chi$ unitary and $\pi \not \cong \pi \times \chi,$ the function $L^S(s,\pi,\Ad\times \chi)$ has no poles on the region $\Re(s)\ge \frac{1}{2}.$ However, during the referee process for this paper, a gap in the proof of the main theorem in \cite{H18} was pointed out. In this paper, we first explain how to close this gap, and then extend \cite[Theorem 6.1]{H18} to the finite part $L$-function $L_f(s,\pi,\Ad\times \chi)=L^{S_\infty}(s, \pi, \Ad \times \chi)$:
 \begin{thm}[Theorem \ref{thm4.2}]\label{thm: main for split case}
 Let $\chi$ be a unitary Hecke character of $F^\times\setminus \BA^\times$ and $\pi$ be an irreducible cuspidal automorphic representation of $\GL_3(\BA)$, then the finite part of the adjoint $L$-function $L_f(s,\pi, \Ad\times \chi)$ is holomorphic on the region $\Re(s)\ge 1/2$, except for a simple pole at $\Re(s)=1$ when $\chi$ is non-trivial and 
$\pi \cong \pi \otimes \chi$ $($which forces $\chi$ to be cubic$)$.
 \end{thm}
 
 As mentioned above, to prove Theorem \ref{thm: main for split case}, we need to analyze Ginzburg's local zeta integral at places $v\in S-S_\infty$, and compare the local zeta integral with the corresponding local L-function. More precisely, we need to show that for each place $v\in S-S_\infty$, there exists a ``test" Whittaker function $W_v$ of $\pi_v$ and a ``test" section $f_{s,v}\in I(s,\chi_v)$ such that the local zeta integral $Z^*(W_v,f_{s,v})$ ``detects" all poles of $L(s,\pi_v,\Ad\times \chi_v)$ on the region $\Re(s)\ge 1/2$, i.e., the quotient $Z^*(W_v,f_{s,v})/L(s,\pi_v,\Ad\times \chi_v)$ has no zeros on the region $\Re(s)\ge \frac{1}{2}$. The main obstacle here is that, at this time, we can not
 rule out the existence of 
 a place  $v\in S-S_\infty$  
 such that $\chi_v$ is unramified
 $\pi_v$ is both ramified and non-tempered.  In this case, $\pi_v$ is of the form $\Ind_{B_3}^{\GL_3(F_v)}(||^{\alpha}\otimes \mu_2\otimes ||^{-\alpha})$, where $B_3$ is the upper triangular Borel subgroup of $\GL_3(F_v)$, $\alpha$ is a real number with $0<\alpha<1/2 $, and $\mu_2$ is a ramified character of $F_v^\times$. It turns out that the required ``test" Whittaker function comes from a certain new form of $\pi_v$. More precisely, let $\psi_v$ be an unramified additive character, and let $W_c\in \CW(\pi_v,\psi_v)$ be the Whittaker function associated with a new vector of $\pi_v$ with respect to the group 
$$K_c=\begin{pmatrix}\fo & \fo & \fp^{-c} \\ \fp^c & 1+\fp^c & \fo\\ \fp^c & \fp^c &\fo \end{pmatrix}^\times,$$
where $\fo$ is the ring of integers of $F_v$, $\fp$ is the maximal ideal of $\fo$ and $c$ is the conductor of $\pi_v$ in the sense of \cite{JPSS81}. Here a new vector of $\pi_v$ with respect to $K_c$ is a nonzero vector in the space $\pi_v^{K_c}$. It is known that $W_c$ is unique (up to scalar).   The main local ingredient to attack Theorem \ref{thm: main for split case} is the following Whittaker function formula for $W_c$:
\begin{thm}[See Theorem \ref{thm1.3}]\label{thm: whittaker function formula}
We have $W_c(1)\ne 0$ and 
$$W_c(\diag(\varpi^m,1,\varpi^{-m}))=\frac{q^{-2m}}{t_1-t_1^{-1}}(t_1^{2m+1}-t_1^{-(2m+1)})W_c(1),$$
where $\varpi$ is a fixed uniformizer of $F_v$, $q$ is the number of the residue field of $F_v$ and $t_1=|\varpi|^\alpha$.
\end{thm}
In fact, Theorem \ref{thm1.3} is slightly general than the above Theorem \ref{thm: whittaker function formula}. After Theorem \ref{thm: whittaker function formula}, a careful choice of test section $f_{s,v}\in I(s,\chi_v)$ (See \S\ref{computation of the local zeta integral in a special case} for the details) will give the desired property, i.e., $Z^*(W_c,f_{s,v})/L(s,\pi_v,\Ad\times \chi_v)$ has no zeros in the region $\Re(s)\ge 1/2$. 
 \begin{rmk}
{\rm After \cite{H18}, Buttcane and Zhou \cite[Theorem 2.4]{BuZh} were able to show that the adjoint $L$-function of a Maass form for $\SL_3(\Z)$ must be entire. Their argument relies on a simple comparison between where the Gamma factor can have poles and where the Riemann zeta function can have zeros. Our result gives an extension of theirs to Maass forms for congruence subgroups of $\SL_3(\Z).$}
\end{rmk}
\begin{rmk}
{\rm Theorem \ref{thm: main for split case} shows that $L_f(s,\pi,\Ad)$ has no poles in the region $\Re(s)\ge \frac{1}{2}$. To prove Conjecture \ref{conj: adjoint GL3 is entire} in the case $n=3$, one still needs to analyze Ginzburg's local zeta integral $Z(W_v,f_{s,v})$ for any archimedean place $v$. In this direction, recently F. Tian is able to show the meromorphic continuation, local functional equation at archimedean places in \cite{Ti18}. Moreover, the local gamma factors for principal series representations at archimedean places is explicitly computed in \cite{Tianthesis, Ti18c}.  }
\end{rmk}

The second part of this paper is 
devoted to the nonsplit case, i.e., when $H=\SU_{2,1}$ is a quasisplit unitary group attached to a quadratic extension $E$ of $F$, and $\pi$ is a globally generic, irreducible cuspidal automorphic representation of $\RU_{2,1}(\BA).$   
In this case, we have not performed a careful analysis on local zeta integrals at ramified primes, because there is still a lot to 
say about the simpler question of when the partial $L$ function attached to all finite unramified places will have a pole. 
As mentioned, the adaptation of Ginzburg's integral representation to apply in nonsplit case was carried out in \cite{H12}.
The adaptation of Ginzburg-Jiang's method of ruling out poles was carried out in \cite{H18}, with interesting results. 

The key to Ginzburg and Jiang's result
is a ``first term identity'' relating residues of Eisenstein series attached to 
different parabolics. Using such an 
identity, they showed that if the global 
integral from \cite{G} has a double pole, then a global integral involving an Eisenstein series on the 
other maximal parabolic subgroup of $G_2$ must be nonzero. They then showed that this global integral 
unfolds to a period integral which vanishes on every cuspidal representation of $\GL_3(\BA).$
When this argument is adapted to the nonsplit case the resulting period, which is on $\SU_{1,1},$ does {\it not} vanish on every cuspidal 
representation of $\RU_{2,1}(\BA).$ Rather, it fails to vanish, by earlier work of Gelbart, Rogawski, and Soudry \cite{GeRoSo2}, 
precisely on the image of an endoscopic lift constructed in \cite{RogawskiBook}.
By carefully combining this information 
with information about the relationship 
between Rogawski's liftings and 
the stable base change lift 
of Kim and Krishnamurthy \cite{KimKrishOdd,KimKrishEven}, 
as well as results about the image of 
that lift obtained by Ginzburg, Rallis and Soudry by the method of functorial descent
\cite{GRSBook}, we obtain the following 
main result.

\begin{thm}[see Theorem \ref{thm: main for nonsplit case}]
Take $\pi$ a globally generic irreducible cuspidal automorphic representation of
the quasisplit unitary group $\RU_{2,1}(\A)$ attached to a quadratic extension 
$E$ of $F.$ Let $\chi_{E/F}$ be the quadratic character attached to $E/F$ by class field theory. Then $L^S(s, \pi, \Ad)$ is holomorphic and nonvanishing at $s=1,$
while $L^S(s, \pi, \Ad\times \chi_{E/F})$ can have at most a double pole. More precisely, we have three sets of equivalent conditions. 
\begin{enumerate}
\item 
The following are equivalent: 
\begin{enumerate}
\item $L^S(s, \pi, \Ad\times \chi_{E/F})$ is holomorphic and nonvanishing at $s=1$
\item the stable base change lift, $\sbc(\pi),$ of $\pi$ is cuspidal 
\item $\pi$ is not endoscopic
\item $\pi$ has a nonzero $\SU_{1,1}$ period
\end{enumerate}
\item 
The following are equivalent: 
\begin{enumerate}
\item $L^S(s, \pi, \Ad\times \chi_{E/F})$ has a simple pole at $s=1$
\item $\sbc(\pi)$ is the isobaric sum of a character and a cuspidal representation of $\GL_2(\A_E)$
\item $\pi$ is an endoscopic lift from $\RU_{1,1}(\A) \times \RU_1(\A),$ but not from  $\RU_{1}(\A) \times \RU_{1}(\A) \times \RU_{1}(\A).$
\end{enumerate}
\item 
The following are equivalent: 
\begin{enumerate}
\item $L^S(s, \pi, \Ad\times \chi_{E/F})$ has a double pole at $s=1$
\item $\sbc(\pi)$ is the isobaric sum of three characters of $\A_E^\times,$
\item $\pi$ is an endoscopic lift from $\RU_{1}(\A) \times \RU_{1}(\A) \times \RU_{1}(\A).$
\end{enumerate}
\end{enumerate}
\end{thm}
The twist by $\chi_{E/F}$ is of special importance for at least two reasons. First, 
in the nonsplit case the local zeta integral $Z^*(W_v, f_{v,s})$ is equal 
to $L(s, \pi_v, \Ad \times \chi_{E/F,v}\chi_v)$ 
as opposed to $L(s, \pi_v, \Ad\times \chi_v).$
(See \cite{H12}.) Second, the analogue of \eqref{Ad = RS/zeta} in the nonsplit case relates
$L(s, \pi , \Ad \times \chi)$ to the Asai $L$-function $L(s, \sbc(\pi), \Asai)$ of the stable base change
lift $\sbc(\pi)$ (\cite{KimKrishOdd,KimKrishEven}) of $\pi.$
Recall that $\sbc(\pi)$ is an automorphic representation of $\GL_3(\A_E).$
The twist by $\chi_{E/F}$ is of particular importance because 
for a cuspidal representation $\Pi$ of $\GL_n(\A_E),$ the 
$L(s, \Pi, \Asai) L(s, \Pi, \Asai \times \chi_{E/F})$ is equal to the 
Rankin-Selberg convolution $L$-function of $\Pi$ with the automorphic representation 
of $\GL_3(\A_E)$ obtained by composing $\Pi$ with the nontrivial element of $\Gal(E/F).$
(cf \cite[p.317]{GRSBook}).  
We also obtain results which characterize
 when an adjoint $L$ function twisted 
 by a character other than $\chi_{E/F}$ 
 can have a pole. See theorems \ref{thm:other poles; type3}, \ref{thm:other poles: type 2+1} and proposition \ref{prop: other poles, type 1+1+1}.
 
 Because it relies on an embedding 
of $\SL_3$ into $G_2$,  
Ginzburg's integral representation does not 
generalize to $\GL_n$ for general $n.$ However, 
it was shown in \cite{BumpGinzburg} that a certain global integral 
in the split exceptional group $F_4$ represents 
 the adjoint $L$-function of $\GL_4,$  and 
some evidence was presented in \cite{GH3} that a certain global integral in the exceptional group $E_8$ represents the adjoint $L$-function for $\GL_5.$ In 
both cases, it appears possible to adapt the construction to apply to 
quasisplit unitary groups as well. It is expected similar results could be done for adjoint representations of $\GL_4$ and $\GL_5$ by studying these integral representations.
 
 \subsection{Acknowledgements}
The first named author was supported by NSA grants H98230-15-1-0234 and  H98230-16-1-0125. The second named author was supported by NSFC grant 11801577 and a fellowship from the Pacific Institute for the Mathematical Sciences (PIMS). We thank Jack Buttcane, Jim Cogdell, Clifton Cunningham, Sol Friedberg, Paul Garrett, Dorian Goldfeld, Dihua Jiang, Muthukrishnan Krishnamurthy, Xiaoqing Li, Baiying Liu, David Loeffler, Stephen Miller, Michitaka Miyauchi, Ralf Schmidt, Richard Taylor, and Fan Zhou for stimulating questions, helpful suggestions, and useful discussions. We thank Fangyang Tian for sending us his unpublished preprints \cite{Tianthesis,Ti18c} and for his comments on our Lemma \ref{lemnonvanishing} of a previous draft. We are very grateful to the anonymous referees for their careful reading and many valuable suggestions.

\section{Local zeta integral for the adjoint representation of \texorpdfstring{$\GL_3$}{Lg} and \texorpdfstring{$\RU_{2,1}$}{Lg}}

In this section, we review Ginzburg's local zeta integral \cite{G} for the adjoint representation of $\GL_3$ and prove the local functional equation for these local zeta integrals. Ginzburg's construction was extended to the $\RU_{2,1}$ case in \cite{H12}.  Let $F$ be a local field in this section. If $F$ is a non-archimedean local field, let $\fo$ be the ring of integers of $F$ and let $q$ be the number of the residue field of $F$. We also fix a uniformizer $\varpi$ of $F$ when $F$ is non-archimedean.

The local zeta integrals defined by Ginzburg involve the unique split exceptional group of type $G_2$. We denote this group $G_2$ and realize it as a set of $8 \times 8$ matrices as in \cite{H12}. Let $B=TU$ be the upper triangular Borel subgroup of $G_2$ under the realization in \cite{H12} with torus $T$ and maximal unipotent subgroup $U$.

 The group $G_2$ has two simple roots $\alpha,\beta$, where $\alpha$ is the short root and $\beta$ is the long root. Then the set of positive roots of $G_2$ is $\wpair{\alpha,\beta,\alpha+\beta,2\alpha+\beta,3\alpha+\beta,3\alpha+2\beta}$. We denote the set of all roots by $\Phi$ and the set of positive roots by $\Phi^+.$ For a root $\delta$, we denote $U_\delta$ the corresponding root space of $\delta$ and $\bx_\delta: \BG_a\ra U_\delta$ a choice of isomorphism. We let $X_\delta = d\bx_\delta(1).$ (Here, $d\bx_\delta$ is the differential.) We assume that $X_\alpha = E_{12}+E_{34}+E_{35}-E_{46}-E_{56}-E_{78},$ $X_{\beta}=E_{23}-E_{67},$ and that the family $\{ \bx_\delta: \delta \in \Phi\}$ is chosen as in \cite{G}, where $E_{ij}\in \textrm{Mat}_{8\times 8}(F)$ is the matrix such that its $i$-th row and $j$-th column is 1 and zero elsewhere. (See also \cite{Ree}). For a root $\delta$, denote $w_\delta(t)=\bx_{\delta}(t)\bx_{-\delta}(-t^{-1})\bx_\delta(t)$ and $w_\delta=w_\delta(1)$. Let $h_\delta(t)=w_\delta(t)w_\delta^{-1}$. For $a,b\in F^\times$, we denote $h(a,b)=h_\alpha(ab)h_\beta(a^2b)$. Then $T(F)=\wpair{h(a,b): a,b\in F^\times},$ and $U=\prod_{\delta>0}U_\delta.$ 
   
We briefly recall some facts about $G_2$ and the particular realization of it given in \cite{H12}.  For more details see \cite{H12} and references therein. The group $G_2$ can be realized as the fixed points of an order three automorphism of $\on{Spin}_8.$ The embedding into eight dimensional matrices in \cite{H12} is obtained by embedding into $\on{Spin}_8$ and then projecting to $\SO_8.$ This projection is actually injective on $G_2$. Thus a symmetric bilinear form is preserved. For the realization in \cite{H12} it is the form attached to the matrix $J$ with ones on the diagonal running from top right to lower left, and zeros elsewhere. The standard representation of $\SO_8$ does not restrict to an irreducible representation of $G_2.$ It decomposes as a one dimensional space on which $G_2$ acts trivially and a copy of the seven dimensional ``standard'' representation of $G_2.$ In \cite{H12}, the invariant one dimensional space is spanned by $v_0:={}^t\!\bbm 0&0&0&1&-1&0&0&0\ebm.$ The seven dimensional ``standard'' representation of $G_2$ supports an invariant quadratic form $Q.$ (In \cite{H12}, it's obtained by restricting the form induced by $J$ to the orthogonal complement of $v_0$.) The vectors satisfying $Q(v)=c$ form a single $G_2(F)$-orbit for each $c \in F^\times.$ The stabilizer of such a vector is isomorphic to $\SL_3,$ either over $F,$ or over a quadratic extension depending on $c.$ For example, the stabilizer $H_\rho$ of $v_\rho:={}^t\! \bbm 0&0&1&0&0&\rho&0&0\ebm$ is isomorphic to $\SL_3$ over the smallest extension of $F$ in which $\rho$ is a square. Indeed, if $\rho = \tau^2$ then 
 $h_\beta(\frac 1{2\tau})\bx_\alpha(\frac{1}{2\tau})\bx_{-\alpha}(-\tau).v_\rho = {}^t\!\bbm
 0&0&0&-\tau &-\tau&0&0&0
 \ebm.$
(In fact, $\bx_\alpha(\frac{1}{2\tau})\bx_{-\alpha}(-\tau).v_\rho = {}^t\!\bbm 0&0&0&-\tau &-\tau&0&0&0\ebm,$ and then acting by $h_\beta(\frac1{2\tau})$ has no effect. The reason for including $h_\beta(\frac1{2\tau})$ will be clear in a moment.) The stabilizer of this vector in $\f g_2$ is the image of the embedding $i:\f{sl}_3\to \f g_2$ given by 
 $$
 i\bpm a&b&c\\d&e&f\\ g&h&-a-e \epm = 
 \bpm 
a+e &&&&&-f&-c&\\ 
&a&b&&&&&c\\ 
&d&e&&&&&f\\ 
&&&0&&&&\\ 
&&&&0&&&\\ 
-h&&&&&-e&-b&\\ 
-g&&&&&-d&-a&\\ 
&g&h&&&&&-a-e\\ 
 \epm.
 $$
 A general element of $\f h_\rho$ is given 
 by 
 $$\bpm
 T_1&a&-\rho e&d&d&e&f&0\\
\rho a&T_1&-\rho d&\rho e&\rho e&d&0&-f\\
h&l&0&a&a&0&-d&-e\\
-\rho l&-h&\rho a&0&0&-a&-\rho e&-d\\
-\rho l&-h&\rho a&0&0&-a&-\rho e&-d\\
-\rho h&-\rho l&0&-\rho a&-\rho a&0&\rho d&\rho e\\
k&0&\rho l&h&h&-l&-T_1&-a\\
0&-k&\rho h&\rho l&\rho l&-h&-\rho a&-T_1\epm.$$
 Conjugating by 
  $h_\beta(\frac1{2\tau})\bx_\alpha(\frac{1}{2\tau})\bx_{-\alpha}(-\tau)$
  yields the image under the injection $i$ of the matrix 
  \begin{equation}
\bpm
-a \tau  + T_{1} & e \tau  - d & -\frac{f}{2 \, \tau } \\
2 \, l \rho - 2 \, h \tau  & 2 \, a \tau  & -e \tau  - d \\
-2 \, k \tau  & 2 \, l \rho + 2 \, h \tau  & -a \tau  - T_{1}
\epm,
  \label{cm}\end{equation}(this is where  $h_\beta(\frac1{2\tau})$ is needed)
  which is a general element of $\SL_3(F)$ if $\tau \in F.$ 
  If $\tau \notin F,$ then \eqref{cm} is a general element 
  of 
  $$\left\{X \in \f{sl}_3(F(\tau)): X\bspm &&-1\\&1&\\-1&& \espm 
  + \bspm &&-1\\&1&\\-1&& \espm  {}^t\ol{X} = 0\right \},$$
  where $\overline{\phantom{F}}$ denotes the action of the nontrivial element of $\Gal(F(\tau)/F).$ This is the Lie algebra of a quasisplit special unitary group attached to the matrix $\bspm &&-1\\&1&\\-1&& \espm.$ In \cite{H12}, the group $\rm U_{2,1}$ is defined (relative to a choice of quadratic extension) using the matrix $\bspm &&1\\&1&\\1&& \espm,$ i.e., 
  $$\RU_{2,1}=\wpair{g\in \GL_3(F(\tau)): {}^t\!\bar g\bspm &&1\\&1&\\1&& \espm g=\bspm &&1\\&1&\\1&& \espm. }$$
  For consistency with \cite{H12} we compose conjugation  by $\diag(1,1,-1)$ in $\GL_3$ with $i$ followed conjugation by  $(h_\beta(\frac1{2\tau})\bx_\alpha(\frac{1}{2\tau})\bx_{-\alpha}(-\tau))^{-1}$ in $G_2$ to obtain an injection ${\rm SU}_{2,1} \into G_2$ in the case $\tau \notin F.$ In the case $\tau \in F$ we obtain an injection $\SL_3 \into G_2,$ which is slightly different from the one used in \cite{G}.
 
 Let $P=MU^\alpha$ be the parabolic subgroup of $G_2$ such that $U_\alpha$ is contained in the Levi $M\cong \GL_2$ of $P$. Here $U^\alpha$ is the unipotent radical of $P$, which is the product of  the root subgroups of $\beta,\alpha+\beta,2\alpha+\beta,3\alpha+\beta,3\alpha+2\beta$. Then $B_\rho:=P \cap H_\rho$ is a Borel subgroup of $H_\rho.$ We denote its unipotent radical $U_\rho$ and its maximal torus $T_\rho.$

\subsection{Induced representations}
Let  $N_{2,\rho}= U_\rho \cap w_\beta P w_\beta^{-1}.$ Then the image of $N_{2,\rho}(F)$ in $\GL_3(F(\tau))$ is 
 $$ \wpair{\begin{pmatrix}
1&r\tau&t\tau +\frac{r^2\rho}2\\
&1&r\tau \\&&1\end{pmatrix},r,t\in F}.$$
In the case $\tau \in F$, this simplifies to 
$ \wpair{\begin{pmatrix}
1&r&t\\
&1&r\\&&1\end{pmatrix},r,t\in F}.$ This is a little different from \cite{G}, where 
$N_2$ is $ \wpair{\begin{pmatrix}
1&r&t\\&1&-r\\&&1\end{pmatrix},r,t\in F}.$ The reason for the difference is the extra conjugation by $\diag(1,1,-1).$

 The Levi subgroup $M$ is generated by elements $\bx_\alpha(r), \bx_{-\alpha}(r), h(a,b)$. Note that $h_\alpha(r)=h(1/r,r^2).$ We consider the isomorphism $M\ra \GL_2(F)$ defined by 
$$\bx_\alpha(r)\mapsto \begin{pmatrix}1& r\\ &1 \end{pmatrix},$$
$$h(a,b)\mapsto \begin{pmatrix}ab &\\ & a \end{pmatrix}.$$
For $m\in M$, we define $\det(m)$ using the isomorphism $M\cong \GL_2$. 

 Let $\delta_P$ be the modulus character of $P$. One can check that $\delta_P(m)=|\det(m)|^3$ for $m\in M$. Let $\chi$ be a character of $F^\times$, we consider the normalized induced representation \label{defn of I(s,chi)}
$$I(s,\chi)=\Ind_{P}^{G_2}(\chi_{s-1/2}),$$
where $\chi_s$ is the character of $M$ defined by $\chi(\det(m))\delta_P^s(m)$.  Note that 
$$\chi_s(h(a,b))=\chi(a^2b)|a^2b|^{3s}.$$

Let $\tilde w=w_\beta w_\alpha w_\beta w_\alpha w_\beta$. Then $\tilde w$ represents the unique Weyl element such that $\tilde w(\alpha)>0$ but $\tilde w(U^\alpha)$ is in the opposite of $U$. Consider the standard intertwining operator 
$$M_{\tilde w}: I(s,\chi)\ra I(1-s,\chi^{-1})$$
which is defined by 
$$M_{\tilde w}(f_s)(g)=\int_{U^\alpha}f_s( \tilde wug)du.$$
By the general theory of intertwining operators, $M_{\tilde w}$ is absolutely convergent for $\Re(s)\gg0$ and can be meromorphically continued to all $s\in \BC$.
\subsection{An exact sequence for induced representations of \texorpdfstring{$G_2$}{Lg}}

For each $\rho \in F^\times,$ we have the double coset decomposition $G_2=Pw_\beta H_\rho\cup P H_\rho$, see \cite[Lemma 3]{H12}. The subgroup $T_0:= \wpair{h(a,1): a\in \GL_1}$ of $T$ is contained in $H_\rho$ for all $\rho$ and is the maximal $F$-split subtorus of $T_\rho$ when  $\rho$ is not a square.  We also have 
$$H_\rho\cap w_\beta P w_\beta^{-1}=N_{2,\rho}\cdot T_0,$$ for all $\rho.$ See \cite[pp. 198-199]{H12}.  Moreover we have the relation 
\begin{equation}\label{eq2.1}
w_\beta h(a,1)=h(1,a)w_\beta.
\end{equation}
On the other hand, recall that $H_\rho\cap P$ is the Borel subgroup $B_\rho$ of $H_\rho$. Mackey's theory gives us an exact sequence of $H_\rho$-modules:
\begin{equation}\label{eq2.2}0\ra \ind_{N_{2,\rho}\cdot T_0}^{H_\rho}(\chi'_s)\ra I(s,\chi)\ra \textrm{n-}\Ind_{B_\rho}^{H_\rho}(\chi_{s})\ra 0,  \end{equation}
where $\ind$ means compact induction, $\textrm{n-}\Ind$ means non-normalized induction, $\chi'_s$ is the character on $N_{2,\rho}\cdot T_0$ defined by 
$$\chi_s'(nh(a,1))=\chi_s(h(1,a))=\chi(a)|a|^{3s},\qquad (n \in N_{2, \rho}, \ a \in \GL_1)$$
and $\chi_s$ is viewed as character on $B_\rho$ by restriction, since $B_\rho \subset P.$  Here the embedding $$ \ind_{N_{2,\rho} \cdot T_0}^{H_\rho}(\chi'_s) \incl I(s,\chi)$$ is defined as follows. For $f_s\in  \ind_{N_{2,\rho} \cdot T_0}^{H_\rho}(\chi'_s)$, then the corresponding $\tilde f_s\in  I(s,\chi)$ is defined by  \label{def of tilde f s}
$$\tilde f_s( mu w_\beta g)=\chi_s(m) f_s(g),$$
where $m\in M, u\in U^\alpha, g\in H_\rho$, and 
$$\tilde f_s(h)=0, \textrm{ if } h\notin Pw_\beta  H_\rho.$$
By \eqref{eq2.1}, $\tilde f_s$ is well-defined.
\subsection{The local zeta integrals}
For $\rho$ in $F^\times$ fix $\tau$ with $\tau^2=\rho,$ and let $\wt H_\rho$ be $\GL_3$ when $\rho$ is a square 
and the quasi-split unitary group ${\rm U}_{2,1}$ otherwise.
Let $\pi$ be an irreducible generic representation of $\wt H_\rho(F)$ and $\chi$ be a quasi-character of $F^\times$. Let $\psi$ be a nontrivial additive character of $F$. 

The maximal unipotent subgroup $U_\rho$ is given by 
$$\left\{\begin{pmatrix} 1&x+y\tau&\frac{(x^2-y^2\rho)}2+w\tau\\&1&-x+y\tau\\ &&1\end{pmatrix}x,y,z\in F\right\}.$$
If $\tau\in F$, this is simply a parametrization of the standard maximal unipotent of $\SL_3$. Let $\psi_\rho:U_\rho(F) \to \C^\times$ be the character given by 
\begin{equation}
\label{psirho}
\psi_\rho\begin{pmatrix} 1&x+y\tau&\frac{(x^2-y^2\rho)}2+w\tau\\&1&-x+y\tau\\ &&1\end{pmatrix}
= \psi(x).
\end{equation}
Note that $\psi_\rho|_{N_{2,\rho}}=1$. Given $W\in \CW(\pi,\psi),f_s\in I(s,\chi)$, the local zeta integral of \cite{G},\cite{H12} is 
$$Z(W,f_s)=\int_{N_{2,\rho}\setminus H_\rho(F)}W(g)f_s(w_\beta g)dg.$$
Observe that if $\tilde f_s\in  \Ind_P^{G_2}(\chi_s)$ is given in terms of  $f_s\in  \ind_{N_{2,\rho} \cdot T_0}^{H_\rho}(\chi_s)$ as in section 
 \ref{def of tilde f s} then 
 $$Z(W,\tilde f_s)=\int_{N_{2,\rho}\setminus H_\rho(F)}W(g)f_s( g)dg.$$

Theorem 5.1 of \cite{H18} states that  for each $s_0$, there exists a choice of data $W,f_s$ (depending on $s_0$) such that $Z(W,f_s)$ is holomorphic and nonvanishing at $s_0$. However, there is a gap in 
the proof of this theorem in the 
archimedean case, because it is 
never shown that $f_{s_0} \mapsto Z(W, f_{s_0})$ 
is a continuous function of $f_{s_0} \in I(s_0, \chi),$ when 
$s_0$ is outside the domain of 
convergence. In the split case, 
it is shown in \cite{Tianthesis, Ti18}
that $Z$ is, in fact, continuous in 
both of its arguments. The technique, which 
was pioneered in \cite{Soudry-Archimedean}, 
should extend to the non-split case as well. 
In fact, the non-split case is easier, because the 
rank of the maximal split torus is only one. 
Here, we content ourselves with sketching how to close the gap in \cite{H18}. 
\begin{lem}
For any fixed $s_0 \in \C$ and $W$ in the Whittaker model of 
some irreducible unitary representation of $\SU_{2,1}(\R),$
the function $f_{s_0} \mapsto Z(W, f_{s_0})$ is a continuous 
function $I(s_0, \chi) \to \C.$
\end{lem}
\begin{proof}
Factoring the Haar measure on $\SU_{2,1}(\R)$ using the 
Iwasawa decomposition, we 
 write the 
local zeta integral as 
$$
\int_K \int_0^\infty\int_{-\infty}^\infty
f_s(w_\beta \bx_\beta(r)) W(d(t)) \psi(tr) t^{3s-3}
\, dr\, dt\, dk,
$$
where $d(t)= \bspm t&\\&1&\\&&t^{-1}\espm\in \SU_{2,1}(\R).$
Arguing as in \cite[Lemma 1, p.197]{Soudry-Archimedean}, or  \cite[Lemma 4.1]{Ti18}, we show that continuity of $Z$ follows from continuity of the 
inner integral 
$$
\int_0^\infty W(t) t^{3s-3} J f_s(t)\, dt, 
\qquad J f_s(t) : = \int_{-\infty}^\infty f_s(w_\beta \bx_\beta(r)) \psi(tr)\, dr.
$$
Arguing as in \cite{JS-ExteriorSquare}, we may write 
$$W(t) = \sum_{\xi \in X} \varphi_\xi(t) \xi(t)$$
where $X$ is a finite set of finite functions on the maximal split torus, 
and, for each $\xi \in X,$ $\varphi_\xi$ is a Schwartz function on $\R.$
Recall that a finite function on the multiplicative group of positive reals is 
of the form $t\mapsto t^u (\log t)^n$ for some complex number $u$ 
and non-negative integer $n.$ 
Hence it suffices to prove continuity of the mapping 
\begin{equation}
    \label{mapping}
f \mapsto \int_0^{\infty} Jf_s(t) \varphi(t) t^u (\log t)^n \, dt,\end{equation}
for any complex number $u,$ nonnegative integer $n,$ and Schwartz 
function $\varphi.$
Since $Jf_s$ is just the Jacquet integral of an embedded $\SL_2$
we have an asymptotic expansion as in \cite[15.2]{Wallach-RRG2}. 
This can be formulated as follows. 
Let $(z_n)_{n=1}^\infty$ be the sequence of complex numbers
obtained by numbering the elements of 
$\{ 3s_0+2k: k = 0,1,2, \dots\} \cup \{ 1-3s_0+2k: k = 0,1,2 \dots\}$ in increasing order
of real part. Then there is a sequence 
$(a_k)_{k=0}^\infty$ of continuous linear functionals
$I(s_0, \chi)\to \C$ and a sequence $(A_k)_{k=1}^\infty$ of continuous 
functions $I(s_0, \chi) \to (0, \infty)$ such that for any non-negative integer
$N$, 
$$
\left|
Jf_{s_0}(t) - \sum_{k=0}^N a_k(f_{s_0}) t^{z_k} 
\right|\le A_{N+1}(f_{s_0}) t^{z_{N+1}},~ \forall t < 1.
$$
Let $Z(\varphi, u, n)$ be defined by 
$$Z(\varphi, u, n)=\int_0^\infty \varphi(t) t^u (\log t)^n \, dt.$$
 for $u > 0$ and by meromorphic continuation elsewhere. 
 (Notice that $Z(\varphi u, 0)$ is a standard Tate zeta factor, while $Z(\varphi, u, n)$ is its $n$th derivative.)
Then 
$$\int_0^{\infty} Jf_s(t) \varphi(t) t^u (\log t)^n \, dt
- \sum_{k=0}^n a_k(f_{s_0})
Z( \varphi, u+z_k, n) 
+E(f_{s_0}),$$
where
$$
|E(f_{s_0})| \le A_{N+1}(f_{s_0})
\int_0^\infty |\varphi(t)| t^{\Re(u+z_{N+1})} ( \log t)^n \, dt,
$$
provided we choose $N$ sufficiently large to ensure that this last integral is 
convergent. Since each of the functions 
$a_k, k=0, \dots, N$ and $A_{N+1}$ tends to zero with $f_{s_0},$ it is now clear that 
\eqref{mapping} is continuous.
\end{proof}

If $F$ is non-archimedean, we can prove a stronger form of the non-vanishing result.

\begin{lem}\label{lemnonvanishing}
Let $F$ be non-archimedean. The local zeta integral $Z(W,f_s)$ is absolutely convergent for $\Re(s)\gg 0$ and can be meromorphically continued to a rational function of $q^s$. Moreover, there exist choices of data $W,f_s$ such that $Z(W,f_s)$ is a nonzero constant.
\end{lem}
The meromorphic continuation of the local zeta integral $Z(W,f_s)$ at the archimedean places in the split case is proved in \cite{Ti18}.
\begin{proof}
The first statement follows from the Bruhat decomposition and the asymptotic behavior of $W$ on the torus element. We next consider the ``moreover" part. To simplify the notation, we only deal with the split case, i.e., when $H_\rho\cong \SL_3$. Let $K_3^m=(1+\textrm{Mat}_{3\times 3}(\fp^m))\cap \SL_3(F)$ be the standard level $m$ congruence subgroup of $\SL_3$. Recall that $B_3$ is the upper triangular Borel subgroup of $\GL_3(F)$.

In this case $T_\rho \cong (F^\times)^2.$ For each $a \in F^\times,$ the group $T_\rho$ contains the element $h(a,1),$ which is identified
with $\diag(a,1,a^{-1})$ in $\SL_3(F).$ Note that $h(a,b)\in T$ is identified with $\diag(a,b,a^{-1}b^{-1})\in \SL_3(F)$ 

Consider the following function on $\SL_3(F)$:
$$f_{s}^m(g)=\left\{\begin{array}{lll}0,& \textrm{ if } g\notin (B_3\cap \SL_3(F))K_3^m\\ \chi_s'(nh(a,1))\phi_1(b)\phi_2(x),  & \textrm{ if } g=nh(a,1)n(x)h(1,b) k, n\in N_2, k\in K_3^m, \end{array}\right.$$
where $$n(x)=\begin{pmatrix}1& x& 0\\ &1&0\\ &&1 \end{pmatrix}\in \SL_3(F), \phi_1\in \CS(F^\times),\phi_2\in \CS(F).$$
To make $f_s^m$ well-defined, we need to require that $\phi_1,\chi_s'$ are constant on $K_3^m\cap T_0$, and $\phi_2$ is invariant under the translation of $\varpi^m\fo_F$. We have $f_s^m\in \ind_{N_{2,\rho}\cdot T_0}^{H_\rho}\chi_s'$. We now compute $Z(W,\tilde f_s^m).$
We assume $m$ is large enough such that $W(gk)=W(g)$ for $k\in K_3^m$. Then factoring the Haar measure on $H_\rho$ 
using the Iwasawa decomposition yields
\begin{align*}
Z(W, \tilde f_s^m)&=c_1\int_{(F^*)^2\times F} W(h(a,1) n(x) h(1,b))f_s^m( h(a,1)n(x)h(1,b) )\, dx\,\frac{ d^*a\, d^*b}{|a^3b^2|}\\
&=c_1\int_{(F^*)^2\times F} W(n(ax) h(a,b)) \chi_s(h(1,a))\phi_1(b)\phi_2(x)\, dx\,\frac{ d^*a\, d^*b}{|a^3b^2|}\\
&=c_1\int_{(F^*)^2} W(h(a,b)) \check \phi_2(a)\phi_1(b)\chi_s(h(1,a))\, \,\frac{ d^*a\, d^*b}{|a^3b^2|},
\end{align*}
where $c_1=\Vol(B_\rho(\fo)) K_3^m$ and  $\check \phi_2(a)=\int_F \psi(ax)\phi_2(x)dx$. 

Assume $m$ is sufficiently large that $a\mapsto \chi_s(h(1,a))$ is trivial on $1+\fp^m,$ and  choose $\phi_1, \phi_2$ such that $\phi_1=\check \phi_2$ is the characteristic function of $1+\fp^m$, we get 
$Z(W, \tilde f_s^m) =c_1\vol(1+\fp^m)^2W(1).$ Clearly, this is constant and $W$ may be chosen so that it is nonzero.
This concludes the proof.
\end{proof}

\subsubsection{Dependence on $\psi$}
In this section we discuss the dependence of the local Ginzburg zeta integral on the choice of additive character $\psi.$ Let $\wh F$ be the Pontriagin dual of $F$ . Then $\widehat F$ is isomorphic 
to $F,$ but not canonically:
indeed if 
$\psi$ is any fixed nontrivial element of $\wh F,$ then every other element of $\wh F$ is of the form 
$\psi^a(x):=\psi(ax)$ for some $a \in F.$ 
The formula \eqref{psirho} gives an injection 
$\widehat F \into \widehat{U_\rho(F)}.$ 
If $\pi$ is an irreducible $\psi_\rho$-generic representation of $\wt H_\rho(F),$
write $\CW(\pi,\psi_{\rho})$ for the Whittaker model 
of $\pi.$
\begin{lem}
Fix $a, \rho \in F^\times.$ Then 
$$(\psi^a)_\rho(u) = \psi_\rho(h(a,1)uh(a,1)^{-1}), \qquad \forall u \in U_\rho(F).$$
\end{lem}
\begin{proof}
Direct computation.
\end{proof}
\begin{cor}
If $\pi$ is $\psi_\rho$ generic then it is $(\psi^a)_\rho$-generic for every $a,$ and 
left translation by $h(a,1)$ is an isomorphism
$\CW(\pi,\psi_{\rho}) \to \CW(\pi,(\psi^a)_{\rho}).$
\end{cor}

\begin{prop}\label{prop:dependence on psi}
If $f_s \in I(s, \chi), \ W \in \CW(\pi, \psi_\rho),$ and $a \in F^\times,$ then 
there exists $W' \in \CW(\pi, (\psi^a)_\rho)$ such that 
$Z(W', f_s) = \chi_s(h(1,a))^{-1}Z(W, f_s).$
\end{prop}
\begin{proof}
Indeed $W'(g) := W(h(a,1)g)$ is an element of $\CW(\pi, (\psi^a)_\rho).$
A change of variable in the integral defining $Z(W, f_s),$
together with the identity $w_\beta h(a,1) = h(1,a) w_\beta$
 shows that $Z(W', f_s) = \chi_s(h(1,a))^{-1}Z(W, f_s).$
\end{proof}

\subsection{The local functional equation}
In this subsection, we assume that $F$ is a non-archimedean local field. Let $\pi$ be an irreducible generic representation of $\wt H_\rho$. It is known that $\pi|_{H_\rho}$ has finite length, see \cite{GeK}.
\begin{lem}\label{prelfe}
Except for a finite number of $q^{-s}$, we have 
$$\Hom_{H_\rho}( \textrm{n-}\Ind_{B_\rho}^{H_\rho}(\chi_{s}),\pi)=0.$$
\end{lem}
\begin{proof}
By the Frobenius reciprocity law, we have 
$$ \Hom_{H_\rho}(\textrm{n-}\Ind_{B_\rho}^{H_\rho}(\chi_{s}),\pi)=\Hom_{T_\rho}( \chi_s,\pi_{U_\rho}).$$
Since $U_\rho$ is the maximal unipotent subgroup of the upper triangular Borel subgroup of both $H_\rho$ and $\wt H_\rho$, we have $\dim\pi_{U_\rho}<\infty$. The assertion follows. 
\end{proof}

\begin{prop}\label{prop: multiplicity one}
Excluding a finite number of $q^{-s}$, we have 
$$\dim \Hom_{H_\rho}
(I(s,\chi),\pi )\le 1.$$
\end{prop}
\begin{proof}
By the exact sequence (\ref{eq2.2}) and Lemma \ref{prelfe}, it suffices to show that $$\dim\Hom_{H_\rho}(\ind_{N_{2,\rho}\cdot T_0}^{H_\rho}(\chi'_s), \pi)\le 1 $$ except for a finite number of $q^{-s}$.

By Frobenius reciprocity law \cite[Proposition 2.29]{BZ}, we have
\begin{align*}&\Hom_{H_\rho}(\ind_{N_{2,\rho}\cdot T_0}^{H_\rho}(\chi'_s), \pi)\\
=&\Hom_{N_{2,\rho}\cdot T_0}(\chi'_s,\pi)\\
=&\Hom_{T_0}(\chi'_s, \pi_{N_{2,\rho}}),
\end{align*}
where $\pi_{N_{2,\rho}}$ is the Jacquet module. The Jacquet module $\pi_{N_{2,\rho}}$ can be viewed as a representation of $T_0\cdot U_\rho$. Since $N_{2,\rho}$ acts trivially on $\pi_{N_{2,\rho}}$, we know that $\pi_{N_{2,\rho}}$ can be viewed a representation of $T_0\ltimes U_\rho/N_{2,\rho}\cong \GL_1\ltimes F$, where the action of $\GL_1\cong F^\times$ on $ F$ is given by multiplication. As a representation of $F\cong U_\rho/N_{2,\rho}$, $\pi_{N_{2,\rho}}$ is smooth. Denote $\sigma=\pi_{N_{2,\rho}}$ and $V_\sigma$ the space of $\sigma$. Thus we have $\CS(F). V_\sigma=V_\sigma$. From the isomorphism induced by the Fourier transform $\CS(\hat F)\cong \CS(F)$, we get $V_\sigma=\CS(\hat F).V_\sigma$, i.e., $\sigma$ is smooth as a $\CS(\hat F)$-module. Thus by \cite[Proposition 1.14]{BZ} there exists a unique sheaf $\CV$ on $\CS(\hat F)$ such that $\CV_c\cong V_\sigma$, where $\CV_c$ denotes the compact support sections in $\CV$.

The action of $F^\times$ on $\hat F$ has two orbits. Let $\psi$ be a nontrivial additive character of $F$, and let $O=\wpair{\psi_a, a\in F^{\times}}$. Then $O$ is the open orbit of the action of $F^\times$ on $\hat F$, and its complement is the trivial character on $F$. We have the usual short exact sequence
$$0\ra \CV_c(O)\ra \CV_c \ra \CV_c(0)\ra 0, $$
where $0$ denotes the zero character. Consider the element $\psi$ in $O$, which may be 
identified with the character $\psi_\rho$ of $U_\rho$ given in \eqref{psirho}.
The stalk of the sheaf $\CV$ at the point $\psi$ is given by 
$$(\pi_{N_{2,\rho}})_{U_\rho/N_{2,\rho},\psi}=\pi_{U_\rho,\psi}\cong \BC_\psi,$$
since $\pi$ is an irreducible representation of $\wt H_\rho$ and $\psi$ is a generic character of the maximal unipotent subgroup $U_\rho$ of a Borel subgroup in $\wt H_\rho$. The stabilizer of $\CV$ of the point $\psi$ in $\GL_1$ is $\wpair{1}$. Thus by \cite[Proposition 2.23]{BZ}, we get 
$$\CV_c(O)=\ind_{1}^{\GL_1}(\BC_\psi).$$
Similarly, we have $\CV_c(0)=\pi_{U_\rho}$ which has finite dimension. Thus we get the short exact sequence 
$$0\ra \ind_{1}^{\GL_1}(\BC_\psi)\ra \pi_{N_{2,\rho}}\ra \pi_{U_\rho}\ra 0.$$
Since $\pi_{U_\rho}$ has finite dimension, after excluding a finite number of $q^s$, we have 
$$\Hom_{T_0}(\chi'_s, \pi_{N_{2,\rho}}) =\Hom_{\GL_1}(\chi'_s, \ind_{1}^{\GL_1}(\BC_\psi))=\Hom(\chi'_s, \BC_\psi).$$
Since $\chi'_s$ and $\BC_\psi$ have dimension 1, we get 
$$ \Hom(\chi'_s, \BC_\psi)$$
has dimension 1.
\end{proof}

\begin{cor}\label{corlfe}
There exists a rational function $\gamma(s,\pi, \chi,\psi)$ of $q^s$ such that 
$$Z(W, M_{\tilde w}(f_s))=\gamma(s,\pi,\chi,\psi)Z(W,f_s),$$
for all $W\in \CW(\pi,\psi), f_s\in I(s,\chi)$.
\end{cor}
\begin{proof}
Since both $(W,f_s)\mapsto Z(W,f_s)$ and $(W,f_s)\mapsto Z(W,M_{\tilde w}(f_s))$ define a bilinear form on $\Hom_{H_\rho}(I(s,\chi)\otimes \pi,1)$. By Proposition \ref{prop: multiplicity one}, such bilinear form is unique up to a scalar. Thus there is factor $\gamma(s,\pi,\chi,\psi)$ such that $Z(W,M_{\wt w}(f_s))=\gamma(s,\pi,\chi,\psi)Z(W,f_s)$ for all $W\in \CW(\pi,\psi),f_s\in I(s,\chi)$. Let $ W\in \CW(\pi,\psi),\tilde f_s^m\in I(s,\chi)$ be as in the proof of Lemma \ref{lemnonvanishing}. Then $Z(W,\tilde f_s^m)$ is a non-zero constant, say $c$, by Lemma \ref{lemnonvanishing}. Thus we have $\gamma(s,\pi,\chi,\psi)=c^{-1}Z(W,M_{\tilde w}(\tilde f_s^m))$, which is a rational function of $q^s$ by Lemma \ref{lemnonvanishing} again. 
\end{proof}

Based on the relationship between the $L$-functions, it is reasonable to define the 
local gamma factors. 
\begin{defn}\label{defn: local gamma factors}
When $\rho$ is a square $($so $H_\rho \cong \SL_3)$,
let the twisted local adjoint gamma factor be given by 
$$ \gamma(s,\pi,\Ad\times \chi,\psi)=\frac{\gamma(s,(\chi\pi)\times \tilde \pi,\psi)}{\gamma(s,\chi,\psi)},$$    where $\gamma(s,\chi\pi\otimes \tilde \pi,\psi)$ is the local Rankin-Selberg gamma factor defined by Jacquet-Piatetski-Shapiro-Shalika \cite{JPSS83} and $\gamma(s,\chi,\psi)$ is the local gamma factor of Tate. When $\rho$ is a nonsquare $($so $H_\rho$ is a unitary group$)$, let the twisted local $\Ad'$ gamma factor be given by
$$
\gm(s, \pi, \Ad' \times \chi,\psi) = \frac{\gm(s, \sbc(\pi), \Asai \times \chi,\psi)}{\gm(s, \chi, \psi)},
$$
where the denominator is the Tate gamma factor and the numerator is defined by the Langlands-Shahidi method.
\end{defn}
Here $\Ad'$ is defined as in \cite{H18}. Following \cite{H18} we define $\Ad'$ to be $\Ad$ in the split case. 
Then the unramified computations from \cite{H12} show that in both the split and nonsplit cases, 
\begin{equation}
\label{gamma factor formula}
\gm(s, \pi, \chi, \psi)
=\frac{\gm(3s-1, \pi,\Ad' \times \chi, \psi)}
{\gm(3s-2, \chi, \psi), \gm(6s-3, \chi^2, \psi)\gm(9s-5, \chi^3, \psi)},
\end{equation}
where the individual gamma factors on the right hand side are defined as in  \cite{JPSS83}
or the Langlands-Shahidi method.
 (For $\GL_1$ factors in the denominator, either of these other 
 definitions reduces to the one in Tate's thesis.) It is natural to expect that the local gamma factors in Corollary \ref{corlfe} is essentially the same as the local gamma factors defined in Definition \ref{defn: local gamma factors}, i.e., \eqref{gamma factor formula} should be true for all irreducible generic representation $\pi$ of $\wt H_\rho$. 
 \begin{rmk}
 {\rm Over archimedean local field, the local functional equation of Ginzburg's local zeta integral is proved in \cite{Tianthesis, Ti18, Ti18c} recently. Moreover, it is verified in \cite{Tianthesis, Ti18c} that the local gamma factors for principal series representation of $\GL_3(\BR)$ obtained from the local functional equation satisfies \eqref{gamma factor formula}.}
 \end{rmk}

\section{A Whittaker function formula for \texorpdfstring{$\GL_3$}{Lg}}

In this section, we develop a Whittaker function formula for certain ramified induced representation of $\GL_3$ over a $p$-adic field, see Theorem \ref{thm1.3}. This Whittaker function formula will be used to compute the Ginzburg's local zeta integral in a special case, see Proposition \ref{prop: special case}, which is the main ingredient in the proof of  one of our main theorem, Theorem \ref{thm4.2}.

In this section, let $F$ be a non-archimedean local field, $\fo$ the ring of integers of $F$, $\fp$ the maximal ideal of $\fo$, and $\varpi$ a fixed generator of $\fp$. Let $q=|\fo/\fp|$. \label{notion on F}

\subsection{Certain subgroups of \texorpdfstring{$\GL_3$}{Lg}} Let $ B_3=T_3U_3$ be the upper triangular Borel subgroup of $\GL_3(F)$ with diagonal torus $T_3$ and upper triangular unipotent subgroup $U_3$. Let $K=\GL_3(\fo)$. Denote $$w=\begin{pmatrix}&&1\\ &1&\\ 1&& \end{pmatrix},$$
and 
$$\fu(x,y,z)=\begin{pmatrix}1&x&z\\ &1&y\\ &&1 \end{pmatrix}, \hat \fu(x,y,z)=\begin{pmatrix}1&&\\ x&1&\\z&y&1 \end{pmatrix}, \textrm{ for }x,y,z\in F.$$
For a nonnegative integer $n\ge 0$, we consider the subgroup 
$$K_n'=\begin{pmatrix} \fo&\fo&\fo\\ \fo& \fo &\fo\\ \fp^n&\fp^n&1+\fp^n\end{pmatrix}^\times$$
of $\GL_3(F)$. Here and in the following, for a subset $A\subset \textrm{Mat}_{3\times 3}(F)$ which is closed under multiplication, $A^\times$ is used to denote the subset $A^\times:=\wpair{a\in A: a^{-1} \textrm{ exists and }a^{-1}\in A}$ of $A$. It's clear that $A^\times$ is a group, whenever it is nonempty.

Given an irreducible smooth generic complex representation $(\pi,V)$ of $\GL_3(F)$, we consider the $K_n'$-fixed subspace $V'(n)=V^{K'_n}$ of $V$. Denote $c=c_\pi=\min\{n| V'(n)\ne 0\}$, which is called the conductor of $\pi$. By \cite{JPSS81}, the space $V'(c)$ has dimension 1. A nonzero vector of $V'(c)$ is called a new form or new vector of $\pi$. Let $\psi$ be an unramified additive character of $F$. It is known that the epsilon factor $\epsilon(s,\pi,\psi)$ of $\pi$ has the form $C q^{-c_\pi s}$, where $C\in \BC^\times$, see \cite[$\S$5]{JPSS81}. It is worth to note that there was an error in \cite{JPSS81}, which was fixed in \cite{J12} and \cite{Ma13}.

We consider a variant of the above notions. Denote 
$$\epsilon_n=\begin{pmatrix}1&&\\ &&1\\&\varpi^n & \end{pmatrix}\in \GL_3(F),$$
and $K_n=\epsilon_n K_n'\epsilon_n^{-1}$.
In matrix form, we have
$$K_n=\begin{pmatrix}\fo & \fo & \fp^{-n} \\ \fp^n & 1+\fp^n & \fo\\ \fp^n & \fp^n &\fo \end{pmatrix}^\times.$$
Given an irreducible smooth generic complex representation $(\pi,V)$ of $\GL_3(F)$, denote $V(n)=V^{K_n}$, the subspace of $V$ which is fixed by $K_n$. Then $V(n)=\pi(\epsilon_n)V'(n)$. In particular we have $\dim V(c_\pi)=1.$

\begin{rmk}{\rm Let $E/F$ be an unramified quadratic extension of $p$-adic fields and $\RU_{2,1}$ be the unitary group with $3$ variables associated with $E/F$ realized by the matrix $w\in \GL_3$. Let $\fo_E$ be the ring of integers of $E$ and $\fp_E$ be the maximal ideal of $\fo_E$. In \cite{M13},  Miyauchi developed a theory of local new forms for $\RU_{2,1}$ using the group $$\bpm \fo_E& \fo_E& \fp_E^{-n}\\ \fp_E^n & 1+\fp_E^n &\fo_E\\ \fp_E^n & \fp_E^n & \fo_E \epm^\times \cap \RU_{2,1}.$$
The group $K_n$ we choose is inspired by the above group considered by Miyauchi.}
\end{rmk}

\begin{lem}\label{lem1.1}
Let $n\ge 1$.
\begin{enumerate}
\item Let $x,y\in F$. If $\hat \fu(x,0,0)\in B_3 K_n$, then $x\in \fp^n$. Similarly, if $\hat \fu(0,y,0)\in B_3\cdot K_n$, then $y\in \fp^n$.
\item Given $r\in F$, the element $w'(r):=\begin{pmatrix}&&1\\ 1&&\\r &1& \end{pmatrix}$ is not in $B_3\cdot K_n$.
\end{enumerate}
\end{lem}
\begin{proof}
(1) If $\hat \fu(x,0,0)\in B_3 K_n$, then there exists an element $b\in B_3$ such that $b\hat \fu(x,0,0)\in K_n$. Write $$b=\begin{pmatrix}a_1&x_1&z_1\\ &a_2&y_1\\ &&a_3 \end{pmatrix}.$$
Then $$b\hat \fu(x,0,0)=\begin{pmatrix}a_1+x_1x& x_1&z_1\\ a_2x& a_2& y_1\\ &&a_3 \end{pmatrix}.$$
From the condition $b\hat \fu(x,0,0)\in K_n$, we get $a_2\in 1+\fp^n, a_2x\in \fp^n$. Thus $x\in \fp^n$. This proves the first statement. The second one can be proved similarly.

(2) If $w'(r)\in B_3K_n$, there exists $b_3\in B_3$ as above such that $k=bw'(r)\in K_n$. After expanding the matrix of $k$, we get 
$a_3\in \fp^n.$ From the condition $k^{-1}=(w'(r))^{-1}b^{-1}\in K_n$, we can get $a_3^{-1}\in \fo$. This is a contradiction.
\end{proof}
\begin{rmk} {\rm Given $z\in F,$  $\hat \fu(0,0,z)\in B_3K_n$ does not imply $z\in \fp^n$.}
\end{rmk}


\subsection{Induced representation and Whittaker functional}\label{notation I(mu)}
Let $\mu=(\mu_1,\mu_2,\mu_3)$ be a triple of quasi-characters of $F^\times$. We consider the normalized induced representation 
$$I(\mu)=\Ind_{B_3}^{\GL_3}(\mu_1\otimes \mu_2\otimes \mu_3).$$
We assume that the representation $I(\mu)$ is irreducible. By \cite[Theorem 4.2]{BZ77}, the irreducibility of $I(\mu)$ is equivalent to that $\mu_i\mu_j^{-1}\ne |~|^{\pm}$, $1\le i,j\le 3$. It is known that the representation $I(\mu)$ is generic.

An element $f\in I(\mu)$ is a function $f:\GL_3(F)\ra \BC$ such that
$$f(\diag(a_1,a_2,a_3)ug)=\mu_1(a_1)\mu_2(a_2)\mu_3(a_3)|a_1/a_3| f(g),$$ for all $ \diag(a_1,a_2,a_3)\in T_3, u\in U_3, g\in \GL_3(F).$

Let  $U^1\subset U^2\subset\dots $ be a sequence of open compact subgroups of $U_3$ such that $\bigcup_{k=1}^\infty U^k=U_3$. For $f\in I(\mu)$, the sequence of integrals 
$$\int_{U^k} f(wu)\psi^{-1}(u)du, k\ge 1$$
is stable and its limit is independent on the choice of the sequence $\wpair{U^k,k\ge 1}$, see \cite[Proposition 3.2]{Sh} and \cite[Corollary 1.8]{CS}. Denote 
$$\int_{U_3}^{st}f(wu)\psi^{-1}(u)du:=\lim_{k\ra\infty}\int_{U^k} f(wu)\psi^{-1}(u)du,$$
where ``$st$" stands for stable integral. The linear map $f\mapsto \int_{U_3}^{st}f(wu)\psi^{-1}(u)du$ is a Whittaker functional on $I(\mu)$.

In the rest of this section, we assume that $\mu_1,\mu_3$ are unramified and $\mu_2$ is ramified with conductor $c\ge 1$. Let $t_i=\mu_i(\varpi),i=1,3$. By \cite[Theorem 3.4, page 36]{GoJ}, we have $\epsilon(s,I(\mu),\psi)=\prod_{i=1}^3 \epsilon(s,\mu_i,\psi)$. Since $\mu_1,\mu_3$ are unramified, and $\mu_2$ has conductor $c$, we have $\epsilon(s,I(\mu),\psi)=Cq^{-cs}$ for some $C\in \BC^\times$. Thus $c$ is also the conductor of the representation $I(\mu)$. Consequently, we have $\dim I(\mu)^{K_c}=1$.

We consider the following function $f$ on $\GL_3(F)$. We require that $\supp(f)\subset B_3K_c$ and 
$$f(\diag(a_1,a_2,a_3)u k)=\mu_1(a_1)\mu_2(a_2)\mu_3(a_3)|a_1/a_3|, \forall a_1,a_2,a_3\in F^\times, u\in U_3, k\in K_c.$$
The function $f$ is well-defined and right $K_c$-invariant. Thus $f$ is a new form of $I(\mu)$, i.e., $f\in I(\mu)^{K_c}$.

In the following, we fix a nontrivial unramified character $\psi$ of $F$,  and we consider the Whittaker function $W_f$ associated with the new form $f$:
\begin{equation}\label{eq1.1}W_f(g)=\int_{U_3}^{st} f(wug)\psi^{-1}(u)du.\end{equation}

\begin{lem}\label{lem: root killing}
Let $a_i\in F^\times$ with $a_i\in \varpi^{n_i}\fo^\times$ for $i=1,2,3$. If $W_f(\diag(a_1,a_2,a_3))\ne 0$, then $n_1\ge n_2\ge n_3$. 
\end{lem}
\begin{proof}
 For any $x\in \fo$, we have 
$$\diag(a_1,a_2,a_3)\fu(x,0,0)=\fu(a_1a_2^{-1}x,0,0)\diag(a_1,a_2,a_3).$$
Since $\fu(x,0,0)\in K_c$ and $W_f$ is right $K_c$-invariant, we have 
$$W_f(\diag(a_1,a_2,a_3))=\psi(a_1a_2^{-1}x)W_f(\diag(a_1,a_2,a_3)), \forall x\in \fo.$$
If $W_f(\diag(a_1,a_2,a_3))\ne 0$, we must have $\psi(a_1a_2^{-1}x)=1$ for all $x\in \fo$. Since $\psi$ is assumed to be unramified, we have $a_1a_2^{-1}\in \fo$. This shows that $n_1\ge n_2$. Similarly, we can show that $n_2\ge n_3$.
\end{proof}
\begin{thm}\label{thm1.3}
\begin{enumerate}\item We have $W_f(1)\ne 0$. 
\item Let $a\in \varpi^m \fo^\times, b\in \varpi^{-n}\fo^\times$ with $m,n\ge 0$. We have 
$$W_f(\diag(a,1,b))=\frac{q^{-(m+n)}}{t_1-t_3}(t_1t_3)^{-n}(t_1^{m+n+1}-t_3^{m+n+1})W_f(1).$$
\end{enumerate}
\end{thm}

\begin{rmk}{\rm We consider the embedding of $\GL_2$ into $\GL_3$ by $$\bpm a&b\\ c&d \epm \mapsto \bpm a&&b\\ &1&\\ c&&d \epm.$$
Let $\sigma$ be the unramified representation $\Ind_{B_2}^{\GL_2(F)}(\mu_1\otimes \mu_3)$ of $\GL_2(F)$, where $B_2$ is the upper triangular subgroup of $\GL_2(F)$. Let $W^\sigma$ be the unramified Whittaker function of $\sigma$ normalized by $W^\sigma(1)=1$. Then by Shintani's formula for $\GL_2$ \cite{Shin}, the above theorem says that 
$$W_f(\diag(a,1,a^{-1}))=W^\sigma(\diag(a,a^{-1}))W_f(1).$$
This property is what we need in the proof of the holomorphy of the finite part of the adjoint L-function for $\GL_3$ on the region $\Re(s)\ge 1/2$, see Proposition \ref{prop: special case} and Theorem \ref{thm4.2}.  }
\end{rmk}
\begin{rmk}{\rm In \cite{M14}, Miyauchi developed a formula of ramified Whittaker functions for $\GL_n$ (in particular for $\GL_3$), associated with local new forms defined by the group $K_n'$. But Miyauchi's formula does not meet our purpose. It seems that there is no apparent relationship between our formula and Miyauchi's formula, although $K_n$ is conjugate to $K_n'$.  }
\end{rmk}

We will prove Theorem \ref{thm1.3} in the next section. The proof is by brute force computation. It would be interesting to know whether a proof which is more conceptual, and perhaps more amenable to generalization, can be found. 
As a preparation of the proof of the above theorem, we record the following useful lemma
\begin{lem}\label{lem1.4}  In the integrals appearing in the following, we fix the measure $dx$ on $F$ such that $\vol(\fo)=1$, and thus $\vol(\fo^\times)=\vol(\fo)-\vol(\fp)=1-q^{-1}$. 
\begin{enumerate}
\item Let $k$ be an integer, then \begin{equation}\nonumber
\int_{\fp^{k}-\fp^{k+1}}\psi^{-1}(x)dx=\left\{\begin{array}{lll}q^{-k}(1-q^{-1}),  &k\ge 0 \\
-1, & k=-1\\ 0, & k\le -2. \end{array} \right.
\end{equation}
\item Let $i$ be an integer. We have
 \begin{equation}\nonumber
\int_{\fo^\times}\mu_2(1+\varpi^i x)dx=\left\{\begin{array}{lll}1-q^{-1} & i\ge c\\ -q^{-1} & i=c-1, \\ 0, & i<c-1.  \end{array}\right.
\end{equation}
\end{enumerate}
\end{lem}
\begin{proof}
The first statement follows from 
$$\int_{\fp^k}\psi^{-1}(x)dx=\left\{\begin{array}{lll}q^{-k},& k\ge 0, \\ 0, & k<0.\end{array} \right.$$
The second statement can be proved in the same manner if $i>0$. We consider the case $i\le 0$. If $i<0$, we have 
$$1+\varpi^ix=\varpi^i(x+\varpi^{-i}).$$
The set $\wpair{x+\varpi^{-i}:x\in \fo^\times}$ is exactly $\fo^\times$. Thus $\int_{\fo^\times}\mu_2(1+\varpi^ix)du=0$.  If $i=0$, we have 
$$\int_{\fo^\times}\mu_2(1+x)dx=\int_{\fo}\mu_2(1+x)dx-\int_{\fp}\mu_2(1+x)dx=\int_{\fo}\mu_2(x)dx-\int_{\fp}\mu_2(1+x)dx.$$
Here $\int_{\fo}\mu_2(x)dx$ is understood as $\sum_{k\ge 0}\int_{\varpi^k\fo^\times}\mu_2(x)dx$, which is convergent to 0 since $\mu_2$ is ramified. Thus $\int_{\fo^\times}\mu_2(1+x)dx=-\int_{\fp}\mu_2(1+x)dx$, which is $-q^{-1}$ if $c=1$ and $0$ if $c>1$. The assertion follows.
\end{proof}

\subsection{Proof of Theorem \ref{thm1.3}}
We are going to compute 
$$W_f(\diag(a,1,b))=\int_{U_3}^{st} f(wu\diag(a,1,b))\psi^{-1}(u)du, $$
with $|a|=q^{-m}, |b|=q^n$. We first fix the measure $du$. For $u=\fu(x,y,z),x,y,z\in F$, we take $du=dxdydz$, where $dx,dy,dz$ are additive measures on $F$ such that $\Vol(\fo)=1$.

 We can write the above integral as
$$W_f(\diag(a,1,b))=\int_{F^3}^{st} f(w\fu(x,y,z)\diag(a,1,b))\psi^{-1}(x+y)dxdydz. $$
Since $$w\fu(x,y,z)\diag(a,1,b)=\diag( b,1,a)w\fu(a^{-1} x, by, a^{-1}bz),$$
we get 
\begin{equation}\nonumber
W_f(\diag(a,1,b))=\mu_1(b)\mu_3(a)|b/a|\int_{F^3}^{st} f(w\fu(a^{-1}x,by,a^{-1}bz))\psi^{-1}(x+y)dxdydz. \nonumber
\end{equation}
Let $U^k=\wpair{\fu(x,y,z): x,y\in \fp^{-k}, z\in \fp^{-2k}}$. Consider the integral 
\begin{equation}\label{eq1.2}
I^k=I^k(m,n)=\int_{U^k} f(w\fu(a^{-1}x,by,a^{-1}bz))\psi^{-1}(x+y)dxdydz.
\end{equation}
Since $\diag(\fo^\times,1,\fo^\times)\subset K_c$, the right hand side of the above integral only depends on $m,n,k$, which justifies the notation $I^k(m,n)$. We then have
 \begin{equation}\label{eq1.3}W_f(\diag(a,1,b))=q^{n+m}t_1^{-n}t_3^m\lim_{k\ra \infty} I^k.\end{equation}
To compute $I^k$, we will frequently use the following identity
\begin{align}\label{eq1.4}
w\fu(x,y,z)=&\diag(-1/(z-xy),1-(xy)/z ,z) \\
&\cdot \fu\left(\frac{x(z-xy)}{z}, \frac{zy}{z-xy},-z+xy\right) \hat \fu\left(-\frac{y}{z-xy}, \frac{x}{z},\frac{1}{z}\right), \nonumber
\end{align}
if $z\ne 0, z-xy\ne 0$. 
Take $x=y=0,z=\varpi^{-c}$ in Eq.(\ref{eq1.4}), we get 
$$w\fu(0,0,\varpi^{-c})=\diag(-\varpi^c,1,\varpi^{-c})\fu(0,0,-\varpi^{-c})\hat \fu(0,0,\varpi^c).$$
Since $\fu(0,0,\varpi^{-c}),\hat \fu(0,0,\varpi^c)\in K_c$, we can get 
\begin{equation}\label{eq1.5}f(w)=(t_1t_3^{-1}q^{-2})^c\end{equation}
from the above equation.

We now start to compute the integral $I^k$. In the following computation, $k$ is sufficiently large. We have 
\begin{align}
I^k&=\int_{\fp^{-2k}\times \fp^{-k}} \int_{\fp^{-k}} f(w\fu(a^{-1}x,0,a^{-1}b(z-xy) \fu(0,yb,0))) \psi^{-1}(x+y)dy dxdz \label{eq: I^k}\\
&=\int_{\fp^{-2k}\times \fp^{-k}} \int_{\fp^n} f(w\fu(a^{-1}x,0,a^{-1}b(z-xy) \fu(0,yb,0))) \psi^{-1}(x+y)dy dxdz  \nonumber \\
&+\int_{\fp^{-2k}\times \fp^{-k}}\int_{\fp^{-k}-\fp^n}f(w\fu(a^{-1}x,0,a^{-1}b(z-xy) \fu(0,yb,0))) \psi^{-1}(x+y)dy dxdz \nonumber \\
:&=I_1^k+I_2^k. \nonumber
\end{align}
Since $yb\in \fo,$ we have $ \fu(0,yb,0)\in K_c$, for $y\in\fp^{n}$. By changing variable on $z$,  we have
\begin{align*}
I_1^k&=\int_{\fp^{-2k}\times \fp^{-k}} \int_{\fp^n} f(w\fu(a^{-1}x,0,a^{-1}b(z-xy) \fu(0,by,0))) \psi^{-1}(x+y)dy dxdz\\
&=\int_{\fp^{-2k}\times \fp^{-k}} \int_{\fp^n} f(w\fu(a^{-1}x,0,a^{-1}bz )\fu(0,by,0)) \psi^{-1}(x+y)dy dxdz\\
&=q^{-n}\int_{\fp^{-2k}\times \fp^{-k}} f(w\fu(a^{-1}x,0,a^{-1}b z))\psi^{-1}(x)dxdz\\
&=q^{-n}\int_{\fp^{-2k}} \left(\int_{\fp^m} f(w\fu(a^{-1}x,0,a^{-1}b z))dx\right)dz\\
&+q^{-n}\int_{\fp^{-2k}} \left(\int_{\fp^{-k}-\fp^m} f(w\fu(a^{-1}x,0,a^{-1}b z))\psi^{-1}(x)dx\right)dz\\
:&=I_3^k+I_4^k.
\end{align*}
Since for $x\in \fp^m, a^{-1}x\in \fo$, we get 
\begin{align*}
I_3^k&=q^{-m-n}\int_{\fp^{-2k}}f(w\fu(0,0,a^{-1}bz))dz\\
&=q^{-m-n}\left(\int_{\fp^{m+n-c}}f(w\fu(0,0,a^{-1}bz))dz \right.\\
&\left.+ \sum_{i= 1}^{2k+m+n-c}\int_{\fp^{m+n-c-i}-\fp^{m+n-c-i+1}}f(w\fu(0,0,a^{-1}bz))dz \right).
\end{align*}
Since $f$ is $K_c$-invariant, by Eq.(\ref{eq1.5}), we have 
$$\int_{\fp^{m+n-c}}f(wu(0,0,a^{-1}bz))dz=q^{-m-n}(t_1t_3^{-1}q^{-1})^c. $$
Note that for $z\in \fp^{m+n-c-i}-\fp^{m+n-c-i+1}$, we have $ab^{-1}z^{-1}\in \varpi^{c+i}\fo^\times$. By Eq.(\ref{eq1.4}), we have
$$w\fu(0,0,a^{-1}bz)=\diag(-ab^{-1}z^{-1},1,a^{-1}bz) \fu(0,0,-a^{-1}bz)\hat \fu(0,0,ab^{-1}z^{-1}).$$
 For $i\ge 1$, we have $ \hat \fu(0,0,ab^{-1}z^{-1})\in K_c$ and 
$$f(w\fu(0,0,a^{-1}bz))=\mu_1(ab^{-1}z^{-1})\mu_3(a^{-1}bz^{-1})|ab^{-1}z^{-1}|^2=(t_1t_3^{-1}q^{-2})^{c+i}. $$
Thus 
\begin{equation}\label{eq1.6}I_3^k=q^{-2(m+n)}\left((t_1t_3^{-1}q^{-1})^c+(1-q^{-1})\sum_{i=1}^{2k+m+n-c}(t_1t_3^{-1}q^{-1})^{c+i}\right).\end{equation}

We next consider $I_4^k.$
We have
\begin{align*}
I_4^k&=q^{-n}\int_{\fp^{m+n-c}}\left(\int_{\fp^{-k}-\fp^{m}}f(w\fu(a^{-1}x,0,a^{-1}bz))\psi^{-1}(x)dx\right)dz\\
&+q^{-n}\int_{\fp^{-2k}-\fp^{m+n-c}}\left(\int_{\fp^{-k}-\fp^{m}}f(w\fu(a^{-1}x,0,a^{-1}bz))\psi^{-1}(x)dx\right)dz\\
:&=I_5^k+I_6^k.
\end{align*}
For $z\in \fp^{m+n-c}$, we have $a^{-1}bz\in \fp^{-c}$ and thus $\fu(0,0,a^{-1}bz)\in K_c$. We get 
\begin{align*}
I_5^k=q^{-m-2n+c} \int_{\fp^{-k}-\fp^{m}}f(w\fu(a^{-1}x,0,0))\psi^{-1}(x)dx.
\end{align*}
We have the identity 
$$w\fu(a^{-1}x,0,0)=\diag(1,-ax^{-1},a^{-1}x)\fu(0,-a^{-1}x,1)w'(ax^{-1}),$$
see Lemma \ref{lem1.1}(2) for the definition of $w'(r)$. By Lemma \ref{lem1.1}(2), we get $f(w\fu(a^{-1}x,0,0))=0$ and thus $I_5^k=0$.

We next consider $I_6^k$. We have
\begin{align*}
I_6^k&=q^{-n}\sum_{i=1}^{2k+m+n-c}\int_{\fp^{m+n-c-i}-\fp^{m+n-c-i+1}}\\
&\cdot \left(\sum_{j=1}^{k+m}\int_{\fp^{m-j}-\fp^{m-j+1}}f(w\fu(a^{-1}x,0,a^{-1}bz))\psi^{-1}(x)dx\right)dz
\end{align*}
By Eq.(\ref{eq1.4}) again, we have 
$$w\fu(a^{-1}x,0,a^{-1}bz)=\diag(-ab^{-1}z^{-1},1,a^{-1}bz)\fu(a^{-1}x,0,-a^{-1}bz)\hat \fu(0,b^{-1}xz^{-1},ab^{-1}z^{-1} ).$$
For $z\in \fp^{m+n-c-i}-\fp^{m+n-c-i+1},x\in \fp^{m-j}-\fp^{m-j+1}$, we have $ab^{-1}z^{-1}\in \varpi^{c+i}\fo^\times$ and $b^{-1}xz^{-1}\in \varpi^{c+i-j}$. Since $i\ge 1$, we get $\hat \fu(0,0,ab^{-1}z^{-1})\in K_c$. Thus 
$$f( w\fu(a^{-1}x,0,a^{-1}bz))=(t_1t_3^{-1}q^{-2})^{c+i} f(\hat \fu(0,b^{-1}xz^{-1},0)).$$
By Lemma \ref{lem1.1} (1), if $i<j$, we get $f(w\fu(a^{-1}x,0,a^{-1}bz) )=0.$ If $i\ge j\ge 1$, we have 
$$f(w\fu(a^{-1}x,0,a^{-1}bz) )=(t_1t_3^{-1}q^{-2})^{c+i}.$$
Thus, we get 
\begin{align*}
    I_6^k&=q^{-n}\sum_{i=1}^{2k+m+n-c}\int_{\fp^{m+n-c-i}-\fp^{m+n-c-i+1}}(t_1t_3^{-1}q^{-2})^{c+i}\\
&\cdot \left(\sum_{j=1}^{\min\wpair{i,k+m}}\int_{\fp^{m-j}-\fp^{m-j+1}}\psi^{-1}(x)dx\right)dz.
\end{align*}
By Lemma \ref{lem1.4}(1), we have $\int_{\fp^{m-j}-\fp^{m-j+1}}\psi^{-1}(x)=0 $ if $j>m+1$. Then in the above expression, $j$ satisfies the condition $1\le j\le \min \wpair{i,m+1}$. The double sum in the expression of $I_6^k$ then can be divided into 3 parts:
$$\sum_{i=1}^m \sum_{j=1}^i, \quad \sum_{i=m+1}^{2k+m+n-c} \sum_{j=1}^m, \quad \sum_{i=m+1}^{2k+m+n-c}\sum_{j=m+1}^{m+1}.$$
Thus we get 
\begin{align}
I_4^k=I_6^k&=q^{-2m-2n}(1-q^{-1})^2\sum_{i=1}^m(t_1t_3^{-1}q^{-1})^{c+i}\sum_{j=1}^i q^j \label{eq1.7}\\
&+q^{-2m-2n}(1-q^{-1})^2\sum_{i=m+1}^{2k+m+n-c}(t_1t_3^{-1}q^{-1})^{c+i}\sum_{j=1}^m q^j \nonumber\\
&-q^{-m-2n}(1-q^{-1})\sum_{i=m+1}^{2k+m+n-c}(t_1t_3^{-1}q^{-1})^{c+i} \nonumber\\
&=q^{-2m-2n}(1-q^{-1})\sum_{i=1}^m (t_1t_3^{-1}q^{-1})^{c+i}(q^i-1) \nonumber\\
&-q^{-2m-2n}(1-q^{-1}) \sum_{i=m+1}^{2k+m+n-c}(t_1t_3^{-1}q^{-1})^{c+i}. \nonumber
\end{align}
From the above calculation, Eq.(\ref{eq1.6}) and Eq.(\ref{eq1.7}), we get 
\begin{align}\label{eq1.8}
I_1^k&=I_3^k+I_4^k \\
&=q^{-2m-2n}(1-q^{-1})\sum_{i=1}^m (t_1t_3^{-1}q^{-1})^{c+i}(q^i-1) \nonumber \\
&+q^{-2m-2n}(t_1t_3^{-1}q^{-1})^c \nonumber\\
&+q^{-2m-2n}(1-q^{-1}) \sum_{i=1}^{m}(t_1t_3^{-1}q^{-1})^{c+i} \nonumber\\
&=q^{-2m-2n}(t_1t_3q^{-1})^c+q^{-2m-2n}(1-q^{-1})\sum_{i=1}^m(t_1t_3^{-1}q^{-1})^{c+i}q^i, \nonumber
\end{align}
which is independent of $k$ as long as $k$ is sufficiently large. To emphasize the dependence on $m,n$, we will write the above expression as $I_1(m,n)$. 

We next consider 
$$I_2^k=\int_{\fp^{-2k}\times \fp^{-k}}\int_{\fp^{-k}-\fp^n}f(w\fu(a^{-1}x,0,a^{-1}b(z-xy) )\fu(0,by,0)) \psi^{-1}(x+y)dy dxdz.$$
We have 
\begin{align*}
I_2^k&=\int_{\fp^{-2k}\times \fp^{-k}}\int_{\fp^{-k}-\fp^n}f(w\fu(a^{-1}x,by,a^{-1}bz) ) \psi^{-1}(x+y)dy dxdz\\
&=\int_{\fp^{-2k}\times (\fp^{-k}-\fp^{n})}\int_{\fp^{-k}}f(w\fu(a^{-1}x,by,a^{-1}bz) ) \psi^{-1}(x+y)dx dydz\\
&=\int_{\fp^{-2k}\times (\fp^{-k}-\fp^{n})}\int_{\fp^m}f(w\fu(a^{-1}x,by,a^{-1}bz) ) \psi^{-1}(x+y)dx dydz\\
&+\int_{\fp^{-2k}\times (\fp^{-k}-\fp^{n})}\int_{\fp^{-k}-\fp^m}f(w\fu(a^{-1}x,by,a^{-1}bz) ) \psi^{-1}(x+y)dx dydz\\
:&=I_7^k+I_8^k.
\end{align*}
The calculation of $I_7^k$ is the same as that of $I_4^k$, Eq.(\ref{eq1.7}). We only write the result here
\begin{align}
I_{7}^k&=q^{-2m-2n}(1-q^{-1})\sum_{i=1}^n(t_1t_3^{-1}q^{-1})^{c+i}(q^i-1) \label{eq1.9}\\
&-q^{-2m-2n}(1-q^{-1})\sum_{i=n+1}^{2k+m+n-c} (t_1t_3^{-1}q^{-1})^{c+i}. \nonumber
\end{align}

We next compute 
$$I_8^{k}=\int_{\fp^{-2k}\times (\fp^{-k}-\fp^{n})}\int_{\fp^{-k}-\fp^m}f(w\fu(a^{-1}x,by,a^{-1}bz) ) \psi^{-1}(x+y)dx dydz.$$
Using Eq.(\ref{eq1.4}) we get 
\begin{align*}
I_8^k&=\sum_{l=-\infty}^{2k+m+n-c}\int_{\fp^{m+n-c-l}-\fp^{m+n-c-l+1}} \sum_{j=1}^{k+n}\int_{\fp^{n-j}-\fp^{n-j+1}} \sum_{i=1}^{k+m}\int_{\fp^{m-i}-\fp^{m-i+1}}\\
& \cdot \psi^{-1}(x+y)\cdot \mu_1(ab^{-1}(z-xy)^{-1})\mu_2(1-(xy)/z)\mu_3(a^{-1}bz)|ab^{-1}(z-xy)^{-1}|\cdot |a^{-1}bz|^{-1}\\
&\cdot f\left(\hat \fu\left(\frac{-ay}{z-xy},\frac{x}{bz},\frac{a}{bz} \right) \right)dxdydz.
\end{align*}
For $z\in \fp^{m+n-c-l}-\fp^{m+n-c-l+1}, y\in \fp^{n-j}-\fp^{n-j+1}, x\in \fp^{m-i}-\fp^{m-i+1}$, we have 
$$ab^{-1}z^{-1}\in \varpi^{c+l}\fo^\times, b^{-1}xz^{-1}\in \varpi^{c+l-i}, (xy)/z\in \varpi^{c+l-i-j}\fo^\times.$$
For fixed $x,y$, to make sure the integral of $\mu_2(1-(xy)/z)$ with respect to $z$ is non-zero, we need $l\ge i+j-1$, see Lemma \ref{lem1.4} (2) (one can check that the other quantities in $I_8^k$ involving $z$ only depends on the absolute value of $z$). If $l\ge i+j-1$, we have $ay/(z-xy)\in \fp^c, ab^{-1}z^{-1}\in \fp^c, b^{-1}xz^{-1}\in \fp^c$. Moreover, we have 
$$ab^{-1}(z-xy)^{-1}\in \varpi^{c+l}\fo^\times.$$
Thus we get 
\begin{align}\label{eq1.10}
 I_8^k
&= \sum_{j=1}^{k+n}\int_{\fp^{n-j}-\fp^{n-j+1}}\psi^{-1}(y)dy \sum_{i=1}^{k+m}\int_{\fp^{m-i}-\fp^{m-i+1}}\psi^{-1}(x)dx\\
&\cdot \left( \sum_{l=i+j}^{2k+m+n-c}q^{-(m+n)} (1-q^{-1})(t_1t_3^{-1}q^{-1})^{c+l} -q^{-1}q^{-m-n}(t_1t_3^{-1}q^{-1})^{c+i+j-1} \right). \nonumber
\end{align}
Similar to the calculation of $I_6^k$, the right side of  the double sum on $i,j$ in Eq.(\ref{eq1.10}) can be divided into the following 4 parts
$$\sum_{i=1}^m\sum_{j=1}^n, \quad \sum_{i=1}^m \sum_{j=n+1}, \quad \sum_{i=m+1}\sum_{j=1}^n, \quad \sum_{i=m+1}\sum_{j=n+1}.$$
The term corresponding to $\sum_{i=1}^{m}\sum_{j=1}^n$ is 
\begin{align}\label{eq1.11}
&q^{-2m-2n}(1-q^{-1})^2\sum_{i=1}^{m}\sum_{j=1}^n q^{i+j}\left( \sum_{l=i+j}^{2k+m+n-c} (1-q^{-1})(t_1t_3^{-1}q^{-1})^{c+l} -q^{-1}(t_1t_3^{-1}q^{-1})^{c+i+j-1} \right).
\end{align}
The second term is 
\begin{align}\label{eq1.12}
-q^{-2m-2n}(1-q^{-1})\sum_{i=1}^mq^{i+n}\left( \sum_{l=i+n+1}^{2k+m+n-c} (1-q^{-1})(t_1t_3^{-1}q^{-1})^{c+l} -q^{-1}(t_1t_3^{-1}q^{-1})^{c+i+n} \right).
\end{align}
The third term is 
\begin{align}\label{eq1.13}
-q^{-2m-2n}(1-q^{-1})\sum_{j=1}^nq^{i+m}\left( \sum_{l=i+m+1}^{2k+m+n-c} (1-q^{-1})(t_1t_3^{-1}q^{-1})^{c+l} -q^{-1}(t_1t_3^{-1}q^{-1})^{c+i+m} \right).
\end{align}
The last term is 
\begin{align}\label{eq1.14}
 q^{-m-n}\left( \sum_{l=m+n+2}^{2k+m+n-c} (1-q^{-1})(t_1t_3^{-1}q^{-1})^{c+l} -q^{-1}(t_1t_3^{-1}q^{-1})^{c+m+n+1} \right).
\end{align}

We consider the integral $I_2^k=I_7^k+I_8^k$. From Eq.(\ref{eq1.9}-\ref{eq1.14}), the coefficient of the term involes $k$ is 
\begin{align*}
&q^{-2m-2n}(1-q^{-1})\\
&\cdot \left((1-q^{-1})^2 \sum_i^m\sum_j^n q^{i+j}-(1-q^{-1})\sum_{i=1}^mq^{i+n}-(1-q^{-1})\sum_{j=1}^nq^{j+m}+q^{m+n}-1\right),
\end{align*}
which is zero. Thus $I_2^k$ is in fact independent on $k$ as long as $k$ is sufficiently large. We then get 
\begin{align}
I_2^k&=I_7^k+I_8^k \label{eq: I2}\\
&=q^{-2m-2n}(1-q^{-1})\sum_{i=1}^n(t_1t_3^{-1}q^{-1})^{c+i}(q^i-1) \nonumber\\
&+q^{-2m-2n}(1-q^{-1}) \sum_{i=1}^{n}(t_1t_3^{-1}q^{-1})^{c+i} \nonumber \\
&-q^{-2m-2n}(1-q^{-1})^2\sum_{i=1}^{m}\sum_{j=1}^n q^{i+j}\left( \sum_{l=1}^{i+j-1} (1-q^{-1})(t_1t_3^{-1}q^{-1})^{c+l} +q^{-1}(t_1t_3^{-1}q^{-1})^{c+i+j-1} \right) \nonumber\\
&-q^{-2m-2n}(1-q^{-1})\sum_{i=1}^mq^{i+n}\left(- \sum_{l=1}^{i+n} (1-q^{-1})(t_1t_3^{-1}q^{-1})^{c+l} -q^{-1}(t_1t_3^{-1}q^{-1})^{c+i+n} \right) \nonumber \\
&-q^{-2m-2n}(1-q^{-1})\sum_{i=1}^nq^{i+m}\left(- \sum_{l=1}^{i+m} (1-q^{-1})(t_1t_3^{-1}q^{-1})^{c+l} -q^{-1}(t_1t_3^{-1}q^{-1})^{c+i+m} \right) \nonumber \\
&+q^{-m-n}\left( -\sum_{l=1}^{m+n+1} (1-q^{-1})(t_1t_3^{-1}q^{-1})^{c+l} -q^{-1}(t_1t_3^{-1}q^{-1})^{c+m+n+1} \right). \nonumber
\end{align}
To emphasize the dependence on $m,n$, we denote the above expression by $I_2(m,n)$. Let $I(m,n)=\lim_{k\ra \infty}I^k(m,n)$. By Eq.(\ref{eq: I^k}), we have $$I(m,n)=I_1(m,n)+I_2(m,n),$$
where the right hand side was given by Eq.(\ref{eq1.8}) and Eq.(\ref{eq: I2}).

A simple calculation shows that
$$W_f(1)=I(0,0)=I_1(0,0)+I_2(0,0)=(t_1t_3^{-1}q^{-1})^c(1-t_1t_3^{-1}q^{-1}).$$
Since $\mu_1\mu_3^{-1}\ne |~|^{\pm}$ (irreduciblity condition), we have $t_1t_3^{-1}q^{-1}\ne 1$. Thus $W_f(1)\ne 0$.

By a careful symbolic calculation, one can get that 
$$I(m,n)=q^{-2m-2n}\frac{1-(qX)^{m+n+1}}{1-qX}W_f(1),$$
where $X=t_1t_3^{-1}q^{-1}$.
Now Theorem \ref{thm1.3} follows from Eq.(\ref{eq1.3}).

\section{Calculation of Ginzburg's local zeta integral in a special case}\label{computation of the local zeta integral in a special case}
In this section, as in Section \ref{notion on F}, we still let $F$ be non-archimedean local field, $\fo$ be the ring of integers of $F$, $\varpi$ be a uniformizer of $F$, and $q$ be the number of the residue field of $F$. We also fix a nonzero square $\rho,$ and identify the group $H_\rho$ with $\SL_3.$ 

As in section \ref{notation I(mu)}, we now consider the induced representation  $\pi=\Ind_{B_3}^{\GL_3(F)}(\mu_1\otimes\mu_2\otimes \mu_3)$ of $\GL_3(F)$ with, $\mu_1,\mu_3$ unramified, and $\mu_2$ ramified with conductor $c$. We further require that $\mu_1\mu_3=1$. Let $t_1=\mu_1(\varpi)$. Let $W_c\in \CW(\pi,\psi)$ be the Whittaker function defined by \eqref{eq1.1}. We know that $W_c$ is right invariant under the group $K_c\subset \GL_3(F)$. By Lemma \ref{lem: root killing} and Theorem \ref{thm1.3} we have $W_c(1)\ne 0$ and 
\begin{equation}\label{eq3.1}
W_c(\diag(\varpi^m,1,\varpi^{-m}))=\left\{\begin{array}{lll}\frac{q^{-2m}}{t_1-t_1^{-1}}(t_1^{2m+1}-t_1^{-(2m+1)})W_c(1),& m\ge 0,\\ 0, & m<0. \end{array}\right.
\end{equation}

Let $\chi$ be an unramified quasi-character of $F^\times$, $s\in \BC$. Recall that, by the exact sequence (\ref{eq2.2}), the induced representation $I(s,\chi)$ has a subspace $\ind_{N_{2,\rho}\cdot \GL_1}^{H_\rho}(\chi'_s)$. 

For a subset $A\subset F, $ we denote $\ch_A$ the characteristic function of $A$. Consider the following function $f_s^c$ on $H_\rho(F)\cong \SL_3(F)$:
$$f_s^c(g)=\left\{\begin{array}{lll}0,& \textrm{ if } g\notin B_3K_c\cap H_\rho (F),\\ \chi'_s(nh(a,1))\ch_{1+\fp^c}(b)\ch_\fo(x),  & \textrm{ if } g=nh(a,1)n(x)h(1,b) k,   n\in N_{2,\rho}, k\in K_c\cap H_\rho(F), \end{array}\right.$$
where $$n(x)=\begin{pmatrix}1& x& 0\\ &1&0\\ &&1 \end{pmatrix}\in H_\rho(F). $$
Note that for $n\in N_{2,\rho}$, if $nh(a,1)n(x)h(1,b)\in B_3\cap K_c\cap \SL_3(F)$, we can get $a\in \fo^\times, b\in 1+\fp^c, x\in \fo$, and hence $\chi'_s(nh(a,1))\ch_{1+\fp^c}(b)\ch_\fo(x)=1$. Thus $f_s^c$ is well-defined. Note that $f_s^c$ is right $K_c$-invariant. By definition, $f_s^c\in \ind_{N_{2,\rho}\cdot \GL_1}^{H_\rho(F)}(\chi'_s)$ and hence defines an element $\tilde f_s^c\in I(s,\chi)$ by the exact sequence (\ref{eq2.2}).
\begin{prop}\label{prop: special case}
We have $$Z(W_c,\tilde f_s^c)=D_c\frac{1+\chi(\varpi)q^{-(3s-1)}}{(1-t_1^2\chi(\varpi)q^{-(3s-1)})(1-t_1^{-2}\chi(\varpi)q^{-(3s-1)})},$$
where  $D_c=\vol(K_c)\vol(1+\fp^c)W_c(1)$.
\end{prop}
\begin{proof}
We have
$$Z(W_c,\tilde f_s^c)=\int_{N_{2,\rho}\setminus H_\rho(F)}W_c(g)f_s^c(g)dg.$$
Since $\supp(f_s^c)\subset B_3K_c$ and both $W_c$ and $f_s^c$ are left $N_{2,\rho}$-invariant and right $K_c$-invariant, we get 
\begin{align*}
Z(W_c, \tilde f_s^c)&=\vol(K_c)\int_{F^\times} \int_{1+\fp^c}\int_{\fo}W_c(h(a,1)n(x)h(1,b))\chi'_s(h(a,1))|a|^{-3}dx d^*b d^*a\\
&=\vol(K_c)\vol(1+\fp^c)\int_{F^\times}W_c(h(a,1))\chi(a)|a|^{3s-3}d^*a.
\end{align*}
By \eqref{eq3.1}, we get 
\begin{align*}
Z(W_c,\tilde f_s^c)&=D_c\sum_{m\ge 0} \frac{q^{-2m}}{t_1-t_1^{-1}}(t_1^{2m+1}-t_1^{-(2m+1)})(\chi(\varpi) q^{-(3s-3)})^m\\
&=D_c\sum_{m\ge 0}\frac{1}{t_1-t_1^{-1}}(t_1^{2m+1}-t_1^{-(2m+1)})(\chi(\varpi)q^{-(3s-1)})^m\\
&=D_c\frac{1+\chi(\varpi)q^{-(3s-1)}}{(1-t_1^2\chi(\varpi)q^{-(3s-1)})(1-t_1^{-2}\chi(\varpi)q^{-(3s-1)})}.
\end{align*}
This concludes the proof.
\end{proof}

\section{Holomorphy of adjoint \texorpdfstring{$L$}{Lg}-function for \texorpdfstring{$\GL_3$}{Lg}}\label{sec: split main theorem}
In this section, let $F$ be a global field and $\BA$ be the ring of adeles of $F$. Let $\pi=\otimes \pi_v$ be an irreducible cuspidal automorphic representation of $\GL_3(\BA)$. Let $\chi=\otimes\chi_v$ be a unitary Hecke character of $F^\times\setminus \BA^\times$. Then one can consider the twisted adjoint $L$-function $$L(s,\pi,\Ad\times \chi)=\frac{L(s,(\chi\otimes \pi)\times \tilde \pi)}{L(s,\chi)}.$$

For a fixed $\pi$ and $\chi$, let $S=S(\pi,\chi)$ be the finite set of places consisting of all archimedean places and all finite places $v$ such that either $\pi_v$ or $\chi_v$ is ramified. The partial twisted adjoint $L$-function $L^S(s,\pi,\Ad\times \chi)$ is defined by 
$$L^S(s,\pi,\Ad\times \chi)=\prod_{v\notin S}L(s,\pi_v, \Ad\times \chi_v)=\prod_{v\notin S}\frac{L(s,(\chi_v\otimes \pi_v)\times \tilde \pi_v)}{L(s,\chi_v)}.  $$

After the pioneering work of Ginzburg \cite{G}, and Ginzburg-Jiang \cite{GJ}, the following result was obtained:
\begin{thm}[Theorem 6.1 of \cite{H18}]\label{thm: hundley}
 The partial twisted adjoint $L$-function $L^S(s,\pi,\Ad\times \chi)$ has no poles in the half plane $\Re(s)\ge \frac{1}{2}$, except possibly for a simple pole at $\Re(s)=1$ when $\chi$ is non-trivial and $\pi \cong \pi \otimes \chi$. 
(This forces $\chi$ to be cubic.)
Every other pole of the complete $L$-function $L(s,\pi,\Ad\times \chi)$ in $\Re(s)\ge \frac{1}{2}$ is a zero of the Hecke $L$-function $L(s,\chi)$ and a pole of $\prod_{v\in S}L_v(s,\pi_v,\Ad\times \chi_v)$.
\end{thm}

The analogous result for quasisplit unitary groups was also obtained. 
We want to extend the above result to $L_f(s,\pi,\Ad\times \chi)$, where $L_f(s,\pi,\Ad\times \chi)=\prod_{v\notin S_\infty}L(s,\pi_v,\Ad\times \chi_v)$ is the finite part of the $L$-function, where $S_\infty$ is the set of infinite places of $F$. In order to do this, we must treat all of the finite places $v\in S-S_{\infty}$, i.e., either $\pi_v$ or $\chi_v$ 
is ramified. The cases where $\pi_v$ is ramified 
can be split up according to whether $\pi_v$ is tempered or non-tempered. In the non-tempered case, we are able to exploit the classification of unitary 
representations, to say that a representation which is ramified, unitary, and non-tempered is of a fairly specific form. 
In the tempered case, the arguments which we use for $\GL_3$ work equally well in the unitary group case, and we therefore record them in 
this generality.

\begin{lem}\label{lem:tempered}
Let $v$ be a place of $F,$ $\rho$ an element of $F^\times,$ and $\pi_v$ an irreducible generic tempered representation of $H_\rho(F_v).$ 
Then $L(s, \pi_v, \Ad\times \chi_v)$ has no poles in $\Re(s) > 0.$
\end{lem}
\begin{proof} If $H_\rho$ is isomorphic to  $\SL_3$ over $F_v,$ then 
$L(s, \pi_v, \Ad \times \chi) =L(s, \pi_v \times (\wt \pi_v \otimes \chi_v))/L(s, \chi_v).$ 
Since $L(s, \chi_v)^{-1}$ has no poles at all, it suffices to prove that $L(s, \pi_v \times \wt \pi_v \times \chi_v)$ has no poles in $\Re(s) > 0.$ 
This may be deduced from \cite[Proposition 8.4, page 451]{JPSS83}, since $(\wt \pi_v \otimes \chi_v)$ is again tempered. 

If $H_\rho$ is not split over $F_v$, then it determines a quadratic extension field $F_v(\tau)$ of $F_v.$ 
Let $\chi_{F_v(\tau)/F_v}$ denote the corresponding quadratic character. 
Then 
$$L(s, \pi_v, \Ad \times \chi) =L(s, \sbc(\pi_v), \Asai \times (\chi_v\chi_{F_v(\tau)/F_v}))/L(s, \chi_v),$$
where $\sbc$ denotes the stable base change lift of Kim and Krishnamurthy \cite{KimKrishOdd}. 
The twisted Asai $L$-function $L(s, \sbc(\pi_v), \Asai \times \chi_v\chi_{F_v(\tau)/F_v})$ may also be realized 
Asai $L$-function of the twist: $L(s, \sbc(\pi_v)\otimes \wt \chi, \Asai)$
where $\wt \chi$ is any character of $F_v(\tau)^\times$ whose restriction to $F_v^\times$ 
is $\chi_v \chi_{F_v(\tau)/F_v}.$ 
Thus it suffices to show that the stable base change lift of a tempered representation is again tempered, 
and that the Asai $L$-function of a tempered representation has no pole in $\Re(s) > 0. $

The fact that the local stable base change lift of a generic tempered representation is again tempered
is proved for quasisplit $\RU_{2n}$ in Proposition 8.6 of 
\cite{KimKrishEven}.
The argument adapts to $\RU_{2n+1}$ in a straightforward manner, using the results of \cite{KimKrishOdd}.
Holomorphy of the Asai $L$-function for tempered representations in $\Re(s)> 0$  
then follows from Proposition 7.2 of \cite{Shahidi-Plancharel}.
\end{proof}

\begin{prop}\label{prop: calculation of zeta}
Fix a finite place $v\in S-S_{\infty}$ of $F.$ 
Let 
$\pi_v$ be a nontempered irreducible unitary representation of $\GL_3(F_v)$ 
and  $\psi_v$ be an additive character of $F_v$.
Define $\psi_{\rho, v}:U_\rho(F_v) \to \C^\times$ 
by \eqref{psirho} and assume that $\pi_v$ is $\psi_{\rho,v}$-generic.
Let $I(s,\chi_v)$ be the induced representation of $G_2(F_v)$ defined in $\S\ref{defn of I(s,chi)}$. Then there exists a Whittaker function $W_v\in \CW(\pi_v,\psi_{\rho,v})$ and a standard section $f_{s,v}\in I(s,\chi_v)$ such that 
$$L(3s,\chi_v)L(6s-2,\chi_v^2)L(9s-3,\chi_v^3)\frac{Z(W_v,f_{s,v})}{L(3s-1,\pi_v,\Ad\times \chi_v)}$$
has no zeros on the region $\Re(s)\ge \frac{1}{2}$.
\end{prop}
\begin{proof}
By proposition \ref{prop:dependence on psi}, it suffices to treat the case when $\psi$ is also unramified. 
By the classification of unitary representation of $\GL_3$, \cite[$\S$6]{JPSS79}, the representation has the form $\pi_v=\eta_v\otimes \sigma_v$, where $\eta_v$ is a unitary character of $F_v^\times$ and $\sigma_v$ is of the form 
$$\Ind_{B_{3}(F_v)}^{\GL_3(F_v)}(|~|^{\alpha}\otimes \mu_2\otimes |~|^{-\alpha}),$$
where $\mu_2$ is a unitary character of $F_v^\times$ and $\alpha$ is a real number with $0<\alpha<\frac{1}{2}$.

Note that $L(s,\pi_v,\Ad\times \chi_v)=L(s,\sigma_v,\Ad\times \chi_v)$ and for $W_v\in \CW(\pi_v,\psi_v)$, $W_v|_{\SL_3(F_v)}$ is a Whittaker function for $\sigma_v$. Since Ginzburg's integral $Z(W_v,f_{s,v})$ only depends on $W_v|_{\SL_3(F_v)}$, we can assume that $\eta_v=1$, i.e., $\pi_v=\sigma_v$. 

Note that the assumption $v\in S-S_\infty$ implies that either $\mu_2$ is ramified or $\chi_v$ is ramified.

We first consider the case when $\mu_2$ is unramified. If $\chi_v$ is ramified, from the equation $L(s,\pi_v,\Ad\times \chi_v)=L(s,(\chi_v\pi_v)\otimes \wt \pi_v)/L(s,\chi_v)$ and \cite[Theorem 3.1]{JPSS83}, one can check that $L(s,\pi_v,\Ad\times \chi_v)=1$. Thus the Claim follows from the fact that one can find $W_v,f_{s,v}$ such that $Z(W_v,f_{s,v})$ is a nonzero constant, see Lemma \ref{lemnonvanishing}.

We next consider the case when $\mu_2$ is ramified and $\chi_v$ is also ramified. In this case by \cite[Theorem 4.1]{CPS}, one can check that  $$L(s,\pi_v,\Ad\times \chi_v)=L(s,\chi_v\mu_2^{-1}|~|^\alpha)L(s,\chi_v\mu_2^{-1}|~|^{-\alpha})L(s,\chi_v\mu_2|~|^\alpha)L(s,\chi_v\mu_2|~|^{-\alpha}).$$ Note that if $\chi_v\mu_2$ and $\chi_v\mu_2^{-1}$ are ramified, then $L(s,\pi_v,\Ad\times\chi_v)=1$ and thus it has no pole. 
If either  $\chi_v\mu_2$ or $\chi_v\mu_2^{-1}$ is unramified, then $L(s,\pi_v,\Ad\times\chi_v)$ has no pole on the region $\Re(s)\ge \frac{1}{2},$ because  $\chi_v$ and $\mu_2$ are unitary and $\alpha<1/2.$ Hence $L(s,\pi_v,\Ad\times \chi_v)$ has no pole in the region $\Re(s)\ge \frac{1}{2}.$
The assertion then also follows from the nonvanishing of the local Ginzburg's integrals, see Lemma \ref{lemnonvanishing}.

Finally, we consider the case when $\mu_2$ is ramified and $\chi_v$ is unramified. In this case, by \cite[Theorem 4.1]{CPS} again, we have
$$L(s,\pi_v,\Ad\times \chi_v)=L(s,\chi_v)^2L(s,\chi_v|~|^{2\alpha})L(s,\chi_v|~|^{-2\alpha}).$$
By Proposition \ref{prop: special case}, we can take $W_v\in \CW(\pi_v,\psi_v)$ and a standard section $f_{s,v}\in I(s,\chi_v)$ such that 
\begin{align*}Z(W_v,f_{s,v})&=\frac{1+\chi_v(\varpi)q^{-(3s-1)}}{(1-\chi_v(\varpi)|\varpi|^{2\alpha}q^{-(3s-1)})(1-\chi_v(\varpi)|\varpi|^{-2\alpha}q^{-(3s-1)})}\\
&=\frac{1-\chi_v(\varpi)^2q^{-(6s-2)}}{(1-\chi_v(\varpi)q^{-(3s-1)})(1-\chi_v(\varpi)|\varpi|^{2\alpha}q^{-(3s-1)})(1-\chi_v(\varpi)|\varpi|^{-2\alpha}q^{-(3s-1)})}\\
&=\frac{L(3s-1,\chi_v)L(3s-1,\chi_v|~|^{2\alpha})L(3s-1,\chi_v|~|^{-2\alpha})}{L(6s-2,\chi_v^2)}.
\end{align*}
Thus 
\begin{align*}
\frac{Z(W_v,f_{s,v})}{L(3s-1,\pi_v,\Ad\times \chi_v)}&=\frac{1}{L(3s-1,\chi_v)L(6s-2,\chi_v^2)},
\end{align*}
and 
\begin{align*}
&L(3s,\chi_v)L(6s-2,\chi_v^2)L(9s-3,\chi_v^3)\frac{Z(W_v,f_{s,v})}{L(3s-1,\pi_v,\Ad\times \chi_v)}\\
=&L(3s,\chi_v)\frac{L(9s-3,\chi_v^3)}{L(3s-1,\chi_v)}\\
=&\frac{1}{(1-\chi_v(\varpi)q^{-3s})(1+\chi_v(\varpi)q^{-(3s-1)}+\chi_v(\varpi)^2q^{-(6s-2)})}.
\end{align*}
The last expression clearly has no zeros. The assertion follows. 
\end{proof}

\begin{thm}\label{thm4.2}
Let $\chi$ be a unitary Hecke character of $F^\times\setminus \BA^\times$ and $\pi$ be an irreducible cuspidal automorphic representation of $\GL_3(\BA)$, then the finite part of the adjoint $L$-function $L_f(s,\pi, \Ad\times \chi)$ is holomorphic on the region $\Re(s)\ge 1/2$, except for a simple pole at $\Re(s)=1$ when $\chi$ is non-trivial and 
$\pi \cong \pi \otimes \chi$ (which forces $\chi$ to be cubic).
\end{thm}
\begin{proof}
Note that we can write $\pi=|\det|^{\zeta}\pi_0$ with $\zeta\in \BC$ and $\pi_0$ unitary. From the relation $L_f(s,\pi,\Ad\times\chi)=L_f(s,(\chi\otimes\chi)\times \tilde \pi)/L_f(s,\chi)$, we can get $L_f(s,\pi,\Ad\times \chi)=L(s,\pi_0,\Ad\times\chi)$. Replacing $\pi$ by $\pi_0$ if necessary, we can assume that $\pi$ is unitary.

We first show that $L_f(s,\pi, \Ad\times \chi)$ is holomorphic for $\Re(s)\ge 1/2$.

Let $S=S(\pi,\chi)$. Then we can write $L_f(s,\pi,\Ad\times \chi)=\prod_{v\in S-S_\infty}L(s,\pi_v,\Ad \times \chi_v) \cdot L^S(s,\pi,\Ad\times \chi)$. By Theorem \ref{thm: hundley}, it suffices to consider $L(s,\pi_v, \Ad\times \chi_v)$ for every place $v\in S-S_\infty$. Note that each $\pi_v$ is unitary since $\pi$ is cuspidal. Note that if $\pi_v$ is tempered, $L(s,\pi_v, \Ad\times \chi_v)$ has no pole on $\Re(s)>0$ by Lemma \ref{lem:tempered}.

We next assume that $\pi_v$ is non-tempered.  Given a global section $f_s\in I(s,\chi)$, one can consider the normalized Eisenstein series 
 $$E(g,f_s)=L(3s,\chi)L(6s-2,\chi^2)L(9s-3,\chi^3)\sum_{\lambda\in P(F)\setminus G_2(F)}f_s(\lambda g).$$
  Given a cusp form $\varphi$ in the space of $\varphi$, which is assumed to correspond to a pure tensor, and a pure tensor $f_s=\otimes f_{s,v}\in I(s,\chi)$, Ginzburg \cite{G} defined
    the global integral $$Z(\varphi, f_s)=\int_{\SL_3(F)\setminus \SL_3(\BA)}\varphi(g)E(g,f_s)dg,$$ and showed that it is Eulerian: 
  $$Z(\varphi,f_s)=\prod_{v} Z^*(W_v,f_{s,v}),$$
where $Z^*(W_v,f_{s,v})=L(3s,\chi_v)L(6s-2,\chi_v^2)L(9s-3,\chi_v^3)Z(W_v,f_{s,v})$, $W_v$ is the 
$v$-th component of the Whittaker function of $\varphi.$
 In \cite{G}, it is also showed that if $v$ is a place such that $\pi_v, \chi_v$ and $\psi_v$ 
  are all unramified, then $Z^*(W_v,f_{s,v})=L(3s-1,\pi_v,\Ad\times \chi_v)$ for spherical $W_v,f_{s,v}$. 
  (Both \cite{G} and  \cite{H12} 
   only treat the case when $\chi$ is trivial, but the extension to nontrivial $\chi$ is direct.)
  The same identity holds up to an exponential factor at places 
  when $\pi_v$ and $\chi_v$ are unramified but 
  $\psi_v$ is ramified.
  Thus we get 
\begin{align}Z(\varphi,f_s)&=
\chi_s(h(1,a))\prod_{v\in S_\infty}Z^*(W_v,f_{s,v})\cdot \prod_{v\in S-S_\infty}Z^*(W_v,f_{s,v})\cdot L^S(3s-1,\pi,\Ad\times \chi) \label{eq4.1}\\
&=\chi_s(h(1,a))\prod_{v\in S_\infty}Z^*(W_v,f_{s,v})\cdot \prod_{v\in S-S_\infty}\frac{Z^*(W_v,f_{s,v})}{L(3s-1,\pi_v,\Ad\times \chi_v)}\cdot L_f(3s-1,\pi,\Ad\times \chi), \nonumber
\end{align}
for a certain id\`ele $a$ determined by the 
finite 
places $v$ such that $\psi_v$ is ramified
and $\pi_v$ and $\chi_v$ are not.
After the work of \cite{GJ}, it is shown in  \cite{H18} that for a flat section $f_s\in I(s,\chi)$, $Z(\varphi,f_s)$ has no pole on the region $\Re(s)\ge 1/2$ except for a possible simple pole at $\Re(s)=2/3$ which can occur only when $\chi$ is cubic. By Eq.(\ref{eq4.1}), the non-vanishing results of the local zeta integrals $Z^*(W_v,f_{s,v})$ \cite[Theorem 5.1]{H18} and Proposition \ref{prop: calculation of zeta}, we obtained that $L_f(3s-1,\pi,\Ad\times \chi)$ has no pole on the region $\Re(s)\ge 1/2$ except for a possible simple pole at $s=2/3$ in the case when $\chi$ is cubic. 

Thus $L_f(s,\pi, \Ad\times \chi)$ has no pole on the region $\Re(s)\ge 1/2$ except for a possible simple pole at $s=1$ in the case $\chi$ is cubic. 
By Proposition 3.6 of  \cite{JS-EulerProductsII}, combined with \cite[Proposition 8.4, page 451]{JPSS83}, $L_f(s,\pi,\Ad\times \chi )$
has a  pole (which is simple when it exists) at $s=1$ if and only if $\chi\pi \cong \pi.$ 
By Tate's thesis or the original work of Hecke, $L_f(s, \chi)$ has a pole  (which is simple when it exists) at $s=1$ if and only if $\chi$ is 
trivial (in which case $L_f(s, \chi)$ is just the Dedekind zeta function of $F$). 

It follows that $L_f(s,\pi,\Ad\times \chi )$ has a pole at $s=1$ if and only if $\pi \cong \chi\pi$ (which implies $\chi^3=1$), and $\chi$ is nontrivial.
\end{proof}

\renewcommand{\ra}{\rangle}
\renewcommand{\bm}{\begin{matrix}}

\section{Further discussion of poles in the nonsplit case}\label{sec: quasisplit}
For the remainder of the paper we devote our attention to a detailed study of the poles of $L^S(s, \pi, \Ad \times \chi)$ when $\rho$
is a non-square. Here $S$ is a finite set of places, containing all Archimedean places, such that $\pi_v$ and $\chi_v$ 
are unramified for $v \notin S.$ 
By \cite{H18}, theorem 6.4,   $L^S(s, \pi, \Ad \times \chi)$ has no poles in $\Re(s) \ge \frac 12,$ and may have a pole at $s=1$
only if $\chi$ is nontrivial cubic or $\chi$ is quadratic and $\pi$ is 
distinguished with respect to a group $H'_\rho,$ which we may think of as $\SL_2,$ or more suggestively as $\SU_{1,1}$ 
embedded into $\SU_{2,1}.$ See \cite{H18} for details.
(In view of Lemma \ref{lem:tempered}, this information about the poles of $L^S(s, \pi, \Ad \times \chi)$ remains true even if we remove places $v$ such that $\pi_v$ is tempered from $S.$) 
Let $H^\flat_\rho$ denote the quasisplit unitary
 group $\RU_{1,1}$, embedded into $H_\rho$ so that $H_\rho'$ is the derived group. 
 Determinant maps $\RU_{1,1}$ to $\RU_1$ and we can choose a splitting to 
 write $H_\rho^\flat$ as the internal semidirect product of $H_\rho'$ and a
 subgroup isomorphic to $\RU_1.$ 
Factoring the Haar measure on $H_\rho^\flat$ accordingly shows that
 any representation distinguished with respect to $H_\rho'$ must also support the
period 
$$\varphi \mapsto
\iq{H_\rho^\flat}\varphi(h)\eta^{-1}(\det h)\ dh$$
for some character $\eta$ of $\RU_1(\A)$ (in which case we say that $\pi$ is $(H_\rho^\flat,\eta)$-distinguished).
It is proved in \cite{GeRoSo2} that this forces the $L$-packet $\{\!\{ \pi \}\! \}$ of $\pi$ to 
be 
 the image, under the endoscopic transfer(s)
 constructed in \cite{RogawskiBook}, 
 of some 
$L$-packet $\rho = \rho_2\times \rho_1$ of 
$\RU_{1,1} \times \RU_1,$ 
(where $\rho_2$ is an $L$-packet of representations of $\RU_{1,1}(\A)$ and $\rho_1$ is a character of $\RU_1(\A) \approx \A_E^1$).
(The $L$-homomorphism to which this transfer is attached depends on some choices; varying the choice
 varies the packet $\rho_2$ such that $\rho_2\times \rho_1$ lifts to $\{\!\{\pi\}\!\},$ without changing the overall image of transfer.)
It is then proved 
in \cite{GeRoSo1} that if $\pi$ is $(H_\rho^\flat,\eta)$-distinguished, then at least one of the $L$-packets $\rho=\rho_2\times \rho_1$ 
which transfers to $\{\!\{\pi\}\!\}$ satisfies $\rho_1 = \eta.$ 

Thus, any pole of $L^S(s, \pi, \Ad'\times \chi)$ at $s=1,$ when $\chi$ is either trivial or quadratic, indicates that $\pi$ is endoscopic.
In order to complete our treatment of the nonsplit case, we would like to address the case when $\chi$ is cubic, and to 
study poles in the case when $\pi$ is assumed to be endoscopic. Further investigation requires that we study 
Rogawski's transfer(s) in more detail.

\subsection{Weil forms of \texorpdfstring{$L$}{Lg}-groups}

\subsubsection{A technical point regarding \texorpdfstring{$L$}{Lg}-groups}\label{ss: tech pt}
For purposes of discussing Rogawski's transfer(s) the finite Galois form of the $L$-group will not suffice; 
we must consider the Weil form. We briefly explain the reason. 
We may realize the finite Galois form of ${}^L\RU_{2,1}$ as in \cite{H12} and \cite{H18}, as
$\GL_3(\C) \rtimes \Gal(E/F),$ with the nontrivial element of $\Gal(E/F)$ acting 
by $g\mapsto {}_tg^{-1}.$ Here $_tg=J{}^tgJ,$ as in \cite{H12,H18}.
This automorphism of $\GL_3(\C)$ 
preserves the subgroup 
subgroup\begin{equation}
\label{GL2 x GL1 --> GL3}
\left \{ 
\bpm a& &b \\ &t& \\ c&&d \epm 
\right \} \cong \GL_2(\C) \times \GL_1(\C) \subset \GL_3(\C)
\end{equation}
 of $\GL_3(\C),$ and hence we 
 obtain a subgroup 
 of $^L\RU_{2,1}$ which is
 the semidirect product 
 of $\GL_2(\C) \times \GL_1(\C)$ and  $\Gal(E/F).$

The finite Galois form of the $L$-group of $\RU_{1,1}\times \RU_1$ is also a semidirect product of $\GL_2(\C) \times \GL_1(\C)$ and  $\Gal(E/F),$
but it may not be identified
with this subgroup of $^L\RU_{2,1}.$
Indeed, the nontrivial element of  $\Gal(E/F)$ acts on $^L(\RU_{1,1}\times \RU_1)$
by an automorphism of order two.
As described in \cite{BorelCorvallis}, this automorphism must
satisfy certain conditions. The identity component of $^L(\RU_{1,1}\times \RU_1)$ inherits, from its definition via based root data, a choice of split maximal torus and Borel subgroup. We may fix the isomorphism  $^L(\RU_{1,1}\times \RU_1)^0\to \GL_2(\C) \times \GL_1(\C)$ so that they consist of the diagonal and upper triangular elements, respectively.
The automorphism must
map these to themselves
in a manner determined by duality and the action of $\Gal(\ol F/F)$ on the maximal torus of $\RU_{1,1} \times \RU_1.$ 
The last condition, given in \cite{BorelCorvallis}, 1.2, states that 
for a general reductive group
there must be a set $\{ x_\alpha\}$ of representatives 
for the root subgroups attached to the simple roots, whose elements are permuted 
amongst themselves.
In our case, where there is only one simple root, this last condition means that the action on the standard maximal unipotent subgroup of 
$\GL_2(\C)$ must be trivial.
Since $g \mapsto {}_tg^{-1}$ 
does not act trivially on the 
maximal unipotent subgroup of 
\eqref{GL2 x GL1 --> GL3}
it follows that the semidirect 
product of $\GL_2(\C) \times \GL_1(\C)$ and  $\Gal(E/F),$
which sits naturally inside of $^L\RU_{2,1},$
may not  be identified with $^L(\RU_{1,1} \times \RU_1).$ 

\subsubsection{Definition of the Weil forms of the $L$-group}

The Weil form of the $L$-group of $\RU_{2,1}$ is the semidirect product of $\GL_3(\C)$ and the Weil group, 
$W_F$ of $F,$ (see \cite{Tate-NumberTheoreticBackground}) with the action implicit in the semidirect 
product being defined using the canonical mapping $W_F/W_E \to \Gal(E/F).$
Thus, elements of $W_E$ commute with $\GL_3(\C)$ while elements of $W_F\smallsetminus W_E$ act by $g \mapsto {}_tg^{-1}.$
The  Weil form of the $L$-group of $\RU_{1,1} \times \RU_1$ is defined similarly with 
$\GL_2(\C) \times \GL_1(\C)$ replacing $\GL_3(\C).$ 
We may  take the involution of  $\GL_2(\C) \times \GL_1(\C)$ to 
be
$$
(g, a) \mapsto \left( \bpm & 1 \\ -1& \epm {}^tg^{-1} \bpm & -1 \\ 1& \epm, a^{-1}\right)
= (\det g^{-1} \cdot g, a^{-1} ).
$$

 \subsubsection{Weil forms and the Satake parametrization}
 We recall the parametrization of unramified representations in \cite{BorelCorvallis} 
 which is suited to dealing with Weil forms of $L$-groups. 
 It is based on $L$-homomorphisms from the Weil group $W_F \to {}^L\!G.$ 
 We do not need the full definition of  an $L$-homomorphism, only 
 the  notion of an {\it unramified} $L$-homomorphism of $W_F$ for $F$ local 
 nonarchimedean. 

 In this case $W_F$ is a dense subgroup of $\Gal(\ol F/F)$ and comes equipped with a 
 homomorphism $\ord_{W_F}: W_F \to \Z.$ 
An unramified $L$-homomorphism $W_F \to G^\vee(\C) \rtimes W_F$ is a homomorphism 
which sends $w \in W_F$ 
to $(t^{\ord_{W_F}(w)}, w)$ for some semisimple element $t$ of $G^\vee(\C).$ 
Note that $t$ must be $\Gal(\ol F/ F)$-fixed in order for this to be a homomorphism, 
and, since we work up to conjugacy, we may assume $t$ is in $T^\vee(\C).$ 
Thus conjugacy classes of  unramified $L$-homomorphisms are in bijection with Galois-fixed elements of $T^\vee(\C),$ 
which correspond to unramified characters as usual.
 
 \subsubsection{Weil forms and $L$-functions}
 Let $F$ be nonarchimedean and local.
 We briefly recall the definition of $L(s, \varphi)$ for 
 $\varphi: W_F \to \GL(V)$ a finite dimensional representation of $W_F.$ 

The kernel of $\ord_{W_F}$ is a normal subgroup of $W_F$ called the inertia group. We denote it $I_F.$
As $I_F$ is normal, its fixed subspace $V^{I_F}$ is $W_F$-invariant. 
We define
$$
L(s, \varphi) : =  \det(I - q^{-s} \varphi(w)\big|_{V^{I_F}} )^{-1},
 \quad \text{ for } w \in W_F \text{ with }\ord_{W_F}(w) = 1.
$$
(The expression on the right-hand-side is independent of the choice of $w$.)
 
This then permits us to define $L(s, \pi, r)$ for $\pi$ an unramified representation 
and $r$ a finite dimensional representation of $G^\vee(\C) \rtimes W_F.$ 
Indeed $\pi$ is attached to an unramified $L$-homomorphism 
$\varphi_t(w) = (t^{\ord_{W_F}(w)}, w),$ with $t \in T^\vee(\C)^{W_F},$
and $L(s, \pi, r)$ is defined as 
$$L(s, \pi, r)=L(s, r \circ \varphi_t)= \det(I - q^{-s} r(t)r(w)\big|_{V^{I_F}} )^{-1},\quad \text{ for } w \in W_F \text{ with }\ord_{W_F}(w) = 1.$$

\subsection{Adjoint representations}
For $H=\RU_{2,1}$ (resp. $\RU_{1,1}$) the 
action of ${}^L\!H$ on itself by conjugation determines an action 
on $\f{sl}_3(\C)$ (resp. $\f{sl}_2(\C)$) which we denote $\Ad.$ Since 
each $w \in W_F \ssm W_E$ acts on $\GL_3(\C)$ by $g \mapsto {}_tg^{-1},$
it will act on $\f{sl}_3(\C)$ by 
$X \mapsto -{}_tX.$ 
Note that this 
conflicts with the notation of \cite{H12}.

Let $\Ad'$ denote the representation 
where $\GL_3(\C)$ acts by 
conjugation and each $w \in W_F \ssm W_E$ acts by $X \mapsto {}_tX.$
This is the representation denoted $\Ad$ in \cite{H12}.
It can also be described as the twist of $\Ad$ by the quadratic character $\chi_{E/F}$ attached to $E/F.$

\subsection{Base Change and Automorphic Induction}
Base change for the quadratic extension $E/F$ is the functorial lifting attached to the $L$-homomorphism 
$$
\bc:{}^L(\GL_n) = \GL_n(\C) \times W_F \to {}^L(\Res_{E/F} \GL_n) = (\GL_n(\C) \times \GL_n(\C) ) \rtimes W_F
$$
which sends $w \in W_F$ to itself, and $g \in \GL_n(\C)$ to $(g,g) \in (\GL_n(\C) \times \GL_n(\C) ).$
Automorphic induction for the quadratic extension $E/F$ is the functorial lifting attached to the $L$-homomorphism 
$$\AI: {}^L(\Res_{E/F} \GL_n) = (\GL_n(\C) \times \GL_n(\C) ) \rtimes W_F
\to ^L(\GL_{2n})= \GL_{2n}(\C) \times W_F$$
$$
\AI(g_1, g_2) = \bpm g_1&\\ & g_2 \epm, \qquad \AI(w) = \begin{cases}\bpm &I_n \\ I_n& \epm, & w \in W_F \ssm W_E\\ I_{2n} & w \in W_E.
\end{cases}
$$
Both of these cases of functoriality are proved (in greater generality) in \cite{ArthurClozel}.

\subsection{Stable Base Change and its image}
The stable base change lifting of Kim and Krishnamurthy has already been mentioned a couple of times. It 
lifts globally generic automorphic representations of the quasisplit group $\RU_n$ (Elsewhere in the paper, 
we denoted $\RU_3$ by $\RU_{2,1}$ and $\RU_2$ by $\RU_{1,1}$ to emphasize the quasisplit nature. In the case of general $n$ this notation seems cumbersome.) attached to a quadratic extension $E/F$ 
to automorphic representations of $\Res_{E/F} \GL_n.$ The $L$-group of $\RU_n$ is $\GL_n(\C) \rtimes W_F,$
and $W_F\ssm W_E$ acts on $\GL_n(\C)$ by a nontrivial involution which we denote $g \mapsto g^*.$
There is some freedom to the choice of involution by it must preserve the torus and the borel and permute a collection of elements
$\{ x_\alpha\}$ as in \cite[\S 1.2]{BorelCorvallis}. We can 
 take $g^* = {}_tg^{-1}$ when $n$ is odd but not when $n$ is even (cf. \S \ref{ss: tech pt}). In the even case 
 we can take $g^* = d_0{}_tg^{-1} d_0,$ where $d_0$ is a diagonal matrix with alternating $1$'s and $-1$'s on the diagonal.
 
The $L$-group of $\Res_{E/F}\GL_n$ is $(\GL_n(\C) \times \GL_n(\C) ) \rtimes W_F,$ and $W_F\ssm W_E$ acts 
by permuting the factors. 
Stable base change is the functorial lifting which 
corresponds to the $L$-homomorphism 
$$
\sbc(g,w)= (g, {}_tg^{-1}, w).
$$
It is closely related to the lifting from $\RU_n$ to $\Res_{E/F}\RU_n$ which is considered in \cite{RogawskiBook} (where it is 
called ``base change'').
\subsubsection{Asai representations}
The representation of $\GL_n(\C) \times \GL_n(\C)$ on $\Mat_{n\times n}(\C)$ by 
$(g_1, g_2).X = g_1X{}_tg_2$ is irreducible. Thus it has two distinct extensions to 
a representation of $(\GL_n(\C) \times \GL_n(\C))\rtimes \Gal(E/F),$ 
the finite Galois form of the $L$-group of $\Res_{E/F}\GL_n.$ We denote them 
$\Asai^\pm,$ such that $\Asai^{\pm}(w).X = \pm {}_tX$ for $w \in W_F \smallsetminus W_E.$

It is perhaps more conventional to define the Asai representation using the usual transpose 
as opposed to the lower transpose. The two representations are isomorphic, with an 
isomorphism given by $X \mapsto XJ.$ In the usual Asai representation, the sign is
$+.$ 
\subsubsection{Asai of stable base change}
It is readily verified that
$$
\Asai^+ \circ \sbc =\begin{cases} \Ad' \oplus \mathbf{1}, n \text{ odd,}\\
\Ad \oplus \chi_{E/F}, n \text{ even},
\end{cases}
\qquad 
\Asai^- \circ \sbc = \begin{cases}\Ad \oplus \chi_{E/F},& n \text{ odd,}\\
\Ad'\oplus \mathbf{1}, & n \text{ even}
\end{cases}
$$
where $\mathbf{1}$ is the one dimensional trivial representation, and $\Ad$ and $\Ad'$ are defined as in \cite{H18}.
Thus if we let $(-)^k=+$ for $k$ even and $-$ for $k$ odd, then 
\begin{equation}
\label{eq relating Ad,Ad' Asai for U21}
L^S( s, \sbc(\pi), \Asai^{(-)^n}) =L^S(s, \chi_{E/F}) L^S(s, \pi, \Ad) \qquad
L^S( s, \sbc(\pi), \Asai^{(-)^{n+1}}) =  \zeta^S(s) L^S(s, \pi, \Ad').
\end{equation}
\subsubsection{Image of stable base change}
\label{ss:im sbc}
If $\pi$ is a globally generic automorphic representations of $U_n,$
we denote the stable base change lift by $\sbc(\pi).$ 
The image of this lifting has been characterized in \cite{GRSBook}, Theorem 11.2. 
\begin{thm}[Ginzburg-Rallis-Soudry]\label{thm: image of sbc}
An automorphic representation of $\GL_n(\A_E)$ is in the image of $\sbc$ if and only if it is an isobaric sum 
of distinct cuspidal representations $\tau_1 \boxplus \dots \boxplus \tau_r$ of $\GL_{n_1}(\A_E), 
\dots \GL_{n_2}(\A_E)$ such that $L^S(s, \tau_i, \Asai^{(-)^n+1})$ has a pole at $s=1$ for all $i$ (this pole is 
necessarily simple).
\end{thm}
\begin{rmk}
{\rm As explained in \cite{GRSBook}, $L^S(s, \pi, \Asai^+) L^S(s, \pi, \Asai^-)$ is the partial Rankin-Selberg convolution 
$L$-function $L^S(s, \pi \times \pi\circ \Fr)$ of $\pi$ and the representation obtained by composing $\pi$ with the nontrivial 
element $\Fr$ of $\Gal(E/F).$ It follows that at most one of  $L^S(s, \pi, \Asai^+),$ and  $L^S(s, \pi, \Asai^-)$ may have a 
pole at $s=1.$ Moreover, since both are nonvanishing on the line $\Re(s)=1$ by theorem 5.1 of \cite{Sha81}, it follows that if either 
has a pole then $\pi\circ \Fr = \wt \pi.$ This implies that the central character of $\pi$ is trivial on $\A^\times \subset \A_E^\times.$ 
Finally, if $\wt \pi \cong \pi \circ \Fr,$ then either $\Asai^+$ or $\Asai^-$ has a pole.}
\end{rmk}

 \subsection{Rogawski's liftings}

\subsubsection{Two families of $L$-homomorphisms}
We describe two families of $L$-homomorphisms  introduced in \cite{RogawskiBook}.
\begin{prop}[Rogawski, pp. 52-53]
Fix $\mu: E^\times \bs \A_E^\times \to \C^\times$ satisfying $\mu|_{\A_F^\times} = \chi_{E/F},$
and fix $w_0 \in W_F \smallsetminus W_E.$ 
There is a unique $L$-homomorphism $\xi^{(2,1)}_{w_0, \mu}: {}^L(\RU_{1,1} \times \RU_1)\to {}^L\RU_{2,1}$ such that
$$
\xi^{(2,1)}_{w_0, \mu}\left(  \bpm a&b\\ c&d \epm, t \right) 
= \bpm a&&b\\&t&\\c&&d \epm,\qquad \xi_{w_0, \mu}(w_0) = 
\bpm 1&&\\&1&\\&&-1 \epm w_0, 
$$ and 
$$
\xi^{(2,1)}_{w_0, \mu}(w) = 
\bpm \mu(w)&&\\&1&\\&&\mu(w) \epm w, \qquad  w \in W_E.$$
There is a unique $L$-homomorphism 
$\xi_{w_0, \mu}^{(1,1)}:{}^L(\RU_1\times \RU_1 ) \to {}^L\RU_{1,1}$ such that 
$$
\xi^{(1,1)}_{w_0, \mu}\left(  a,b \right) 
= \bpm a&\\&b \epm,\qquad \xi_{w_0, \mu}(w_0) = 
\bpm &-1\\1& \epm w_0, 
$$ and 
$$
\xi^{(1,1)}_{w_0, \mu}(w) = 
\bpm \mu^{-1}(w)&\\&\mu^{-1}(w) \epm w, \qquad  w \in W_E.$$
\end{prop}

\subsubsection{Rogawski's Liftings}
Each of the $L$-homomorphisms $\xi^{(2,1)}_{w_0, \mu}$ determines a conjectural functorial transfer map, taking 
automorphic $L$-packets on $\RU_{1,1}(\A) \times \RU_1(\A)$ to automorphic $L$-packets on $U_{2,1}(\A).$ 
Likewise, each of the $L$-homomorphisms $\xi^{(1,1)}_{w_0, \mu}$ determines a conjectural functorial transfer map, taking 
automorphic $L$-packets on $\RU_{1}(\A) \times \RU_1(\A)$ to automorphic $L$-packets on $\RU_{1,1}(\A).$ 
The existence of these transfer maps is proved in \cite{RogawskiBook}.
Varying the choice of $w_0$ does not change the transfer mapping. Varying the choice of $\mu$ permutes the elements 
of the image. 

An automorphic representation of $\RU_{1,1}(\A)$ or of $\RU_{2,1}(\A)$ is said to be endoscopic if it is in the image of the 
Rogawski liftings.

Applying the lifting attached to $\xi^{(2,1)}_{w_0, \mu}$ to a packet obtained from $\xi^{(1,1)}_{w_0, \mu}$
gives a packet on $\RU_{2,1}(\A).$ This construction is functorial. We describe the associated $L$-homomorphism.
 
Define 
$$
\xi^{(1,1,1)}(a,b,c) = \bpm a&&\\&b&\\&&c \epm, 
\qquad 
\xi^{(1,1,1)}(w) = m(w) w, \qquad m(w) :=\begin{cases} 
I, w \in W_E, \\J, w \notin W_E.
\end{cases}
$$
(Here $I$ is the identity matrix and $J$ is the matrix with ones on the diagonal from top right to 
lower left and zeros elsewhere.)
\begin{lem}
Take $\eta_1, \eta_2, \eta_3$ three automorphic characters of $\RU_1(\A).$
Let $\pi_1$ be the representation of $\RU_{1,1}(\A)$ obtained from $\eta_1 \otimes \eta_3$ using the Rogawski lifting attached to $\xi^{(1,1)}_{w_0, \mu}$ and let $\pi$ be the representation of $\RU_{2,1}(\A)$ obtained from $\pi_1\otimes \eta_2$ using the Rogawski lifting attached to $\xi^{(2,1)}_{w_0, \mu}.$ Then $\pi$ is the weak functorial lift of $\eta_1\otimes \eta_2\otimes \eta_3$ relative to the $L$-homomorphism $\xi^{(1,1,1)}.$
\end{lem}
\begin{rmk}
{\rm Note that $\xi_{w_0,\mu}^{(1,1)}$ and 
$\xi_{w_0,\mu}^{(2,1)}$ depend on the choice of $w_0$ and $\mu,$ but $\xi^{(1,1,1)}$ does not.}
\end{rmk}
\begin{proof}
Let $v$ be a finite unramified place. Each of our representations is 
determined by an unramified $L$-homomorphism from $W_{F_v}$ into 
the relevant $L$-group. Fix $w$ an  element of $W_{F_v}$ of norm $1.$ 
Then each  unramified $L$-homomorphism from $W_{F_v}$ is determined by 
its image on $w.$ 
We regard $W_{F_v}$ as a subgroup of $W_F$ by some choice of 
embedding as in \cite{Tate-NumberTheoreticBackground}.
Let us refer to this image as the Satake parameter of the representation. 
At a split place,  the  Satake parameter of $\eta_i$ is $(t_i,w).$ 
From considering the isomorphism $W_F/W_E \cong \Gal(E/F)$ we see that 
$w\in W_E$ if and only if $E$ splits over $v.$ 
When this is not the case, $t_i$ must be $1$ for all $i.$ 

First assume $v$ is split. 
Then the Satake parameter of 
$\pi_1$ is $\bspm t_1\mu^{-1}(w)&\\ & t_3 \mu^{-1}(w) \espm w.$ 
Hence, the Satake parameter of $\pi_1 \otimes \eta_2$ is 
$(\bspm t_1\mu^{-1}(w)&\\ & t_3 \mu^{-1}(w) \espm , \ t_2)w,$ and that of $\pi$ is 
$$
\bpm t_1\mu^{-1}(w)&&\\&t_2&\\ && t_3 \mu^{-1}(w) \epm
\bpm \mu(w) &&\\&1& \\ && \mu(w) \epm 
 w
 = \bpm t_1&&\\&t_2&\\ && t_3 \epm
 w.
$$

At an inert place,  the Satake parameter of 
$\pi_1$ is $\bspm &-1\\ 1& \espm w,$ that of $\pi_1\times \eta_2$ is 
$(\bspm &-1\\ 1& \espm, 1)w,$ and that of $\pi$ is 
$$
\bpm &&-1\\ &1&\\1&& \epm \bpm 1&&\\&1&\\&&-1 \epm w
= Jw.
$$
\end{proof}
For $\pi_1$ a cuspidal representation of $\RU_{1,1}(\A)$ and $\eta$ a character of $\RU_1(\A)$ we denote the corresponding 
Rogawski lift by $\Rog^{(2,1)}_{w_0,\mu}(\pi_1 \otimes \eta).$ The lift attached to three characters $\eta_1, \eta_2, \eta_3$ is denoted
$\Rog^{(1,1,1)}(\eta_1\otimes \eta_2\otimes \eta_3).$
\subsubsection{Rogawski liftings and stable base change}
It's clear from the formulae for $\xi^{(2,1)}_{w_0, \mu}$ and $\xi^{(1,1,1)}$ that 
$$\sbc(\Rog^{(2,1)}(\sigma\otimes \eta)
=(\sbc(\sigma)\otimes \mu) \boxplus \sbc(\eta), 
\quad
\sbc(\Rog^{(1,1,1)}(\eta_1\otimes \eta_2 \otimes \eta_3))=\sbc(\eta_1)\boxplus \sbc(\eta_2)\boxplus \sbc(\eta_3).$$
\subsubsection{Rogawski liftings and descent}
For globally generic representations, the combination of the Kim-Krishnamurthy lifting and the 
method of descent gives an alternate construction of Rogawski's endoscopic liftings, and a generalization. 
A representation is endoscopic if and only if its stable base change is non-cuspidal. In this situation, descent 
may be applied to each 
summand of the stable base change, and the original representation is the endoscopic lift of the 
collection of representations thus obtained. 
See \cite{GRSBook}, section 11.3. 

Recall that the isobaric summands of the stable base change of a cusp form are all distinct. 
From this we may deduce that if $\sbc(\eta_1), \sbc(\eta_2)$ and $\sbc(\eta_3)$ are not distinct, 
then $\Rog^{(1,1,1)}(\eta_1 \otimes \eta_2\otimes \eta_3)$ is not cuspidal.

These endoscopic liftings have also been studied by the trace formula method in \cite{Mok}.

\subsubsection{Adjoint $L$-functions of Rogawski lifts}
We regard $\Ad$  as an action of $\GL_3(\C) \rtimes W_F$ on $\mathfrak{sl}_3\C$ by composing 
with the canonical projection to the finite Galois form. Thus $w \in W_F$ acts by 
$X \mapsto {}-{}_t\,\!X$ if $w \notin W_E$ and trivially if $w \in W_E.$ 
In this section we consider $\Ad \circ \xi^{(2,1)}_{w_0, \mu}: {}^L(\RU_{1,1} \times \RU_1) \to \GL( \mathfrak{sl}_3(\C)).$
Let 
$$
X\left(
\bpm x_1 & x_2 \\ x_3 & - x_1 \epm, z, \bpm u_1\\ u_2 \epm, \bpm v_1&v_2\epm  
\right)
= \bpm x_1-z & u_1 & x_2 \\ v_1 & 2z & v_2\\ x_3 & u_2 & -x_1-z \epm .
$$
Then for $x \in \mathfrak{sl}_2(\C), z \in \C, u \in \Mat_{2\times 1}(\C), v \in \Mat_{1\times 2}(\C),$
we have 
$$\begin{aligned}
\Ad\circ \xi^{(2,1)}_{w_0, \mu}(g,t). X(x, z,u,v) 
&= X( gxg^{-1} , z, gut^{-1}, tvg^{-1}), & (g \in \GL_2(\C), \ t \in \C^\times)
\\
\Ad \circ \xi^{(2,1)}_{w_0, \mu}(w).X(x,z,u,v) 
&= X(x,z, \mu(w) u , \mu(w)^{-1} v),&  (w \in W_E)\\
\Ad\circ \xi^{(2,1)}_{w_0, \mu}(w_0).X(x,z,u,v)
&= X\left(x, -z, \bpm -v_2\\ v_1\epm , \bpm -u_2& u_1 \epm \right).
\end{aligned}
$$
Note that 
$$\bpm -v_2\\ v_1\epm = \bpm -1 &\\ &1 \epm {}_tv, \qquad \qquad
 \bpm -u_2& u_1 \epm = {}_tu \bpm -1 &\\ &1 \epm
$$
\begin{prop}
Take $\pi_1$ an irreducible automorphic representation of $U_{1,1}(\A)$ 
and $\eta$ an irreducible automorphic representation (necessarily a character) of $\RU_1(\A).$ Let $\pi = \pi_1 \otimes \eta.$  Let $S$ be a finite set of places of $F,$ including all archimedean places and all places where either $\pi_1$ or $\eta$ is ramified. Then 
\begin{equation}\label{Ad of rog21}
\begin{aligned} 
L^S( s, \pi, \Ad \circ \xi^{(2,1)}_{w_0, \mu}) 
&= L^S(s, \pi_1, \Ad) L^S(s, \chi_{E/F})
L^S(s, \mu \otimes \AI\sbc(\pi_1\otimes \eta^{-1}))\\
L^S( s, \pi, \Ad' \circ \xi^{(2,1)}_{w_0, \mu}) 
&= L^S(s, \pi_1, \Ad') \zeta^S(s)
L^S(s,\mu \otimes  \AI\sbc(\pi_1\otimes \eta^{-1}) )
\end{aligned}
\end{equation}
\end{prop}
\begin{proof}
The proofs of the two statements are parallel. We treat only the first. 
From the computations above we see that $\Ad \circ \xi^{(2,1)}_{w_0, \mu}$ is the direct sum of three irreducible 
components, corresponding to the variable $x,$ the variable $z,$ and the pair $(u,v).$ 
These three components give rise to the three factors above: we match local $L$-factors at both 
split and inert unramified finite places. 
We discuss only the third component in detail, as the first two are easier. 
Denote this representation $r_{\mu, w_0}.$ 
If $v$ is split then $\pi_1$ gives a diagonal matrix $\bpm t_1&\\ & t_2 \epm$ 
and $\eta$ gives a nonzero scalar $c.$ 
We take $w \in W_{F_v}$ with $\ord_{W_{F_v}}(w) = 1.$ Since $G$ is split at $v,$
$\Gal( \ol F_v/F_v)$ acts trivially on $T^\vee(\C).$ 
So, the image of $w$ in $W_F$ lies in $W_E.$ With respect to a suitable basis, 
the matrix of the operator
$$r_{\mu, w_0}\left( \left(\bpm t_1&\\ & t_2 \epm, \ c\right), \ w\right)
$$
is the matrix 
$\diag( \mu(w) \frac{t_1} c, \mu(w) \frac{t_2} c, \mu(w)^{-1} \frac{c}{t_1} , \mu(w)^{-1} \frac c{t_2} ).$
Tracing through the definitions, this is exactly the matrix attachd to 
$\mu \otimes \AI\sbc(\pi_1 \otimes \eta^{-1}).$ 

If $v$ is inert then $\pi_1$ gives a diagonal matrix which is stable under the action of the
Galois group, i.e., of the form $\bpm t&\\ &t^{-1} \epm,$ while $\eta_v$ 
(an unramified character of a compact group)
must be trivial.
Write $\varphi_{t}$ for the corresponding unramified $L$-homomorphism 
$W_F \to G^\vee(\C) \rtimes W_F.$ 
For $w \in W_F$ with $\ord_{W_F}(w) = 1,$ we have 
$$
r_{w_0, \mu}\circ \varphi_{t}(w).( u, v) = \left(\mu(ww_0^{-1}) \bpm -t &\\ &t^{-1} \epm {}_tv, \ \ \ 
\mu^{-1}(ww_0^{-1})\ {}_tu \bpm -t^{-1} &\\ &t \epm\right)
$$
Thus, the matrix of the operator $r_{w_0, \mu}\circ \varphi_{t}(w)$ relative to a suitable choice of basis 
is 
$$
\bpm 0&0&0&-\mu(ww_0^{-1})t\\ 0&0& \mu(ww_0^{-1})t^{-1}\\
0&-\mu^{-1}(ww_0^{-1})t^{-1} &0&0\\
\mu^{-1}(ww_0^{-1})t&0&0&0
\epm,
$$
and the relevant $L$-factor is 
$$
(1+t^2 q^{-2s})(1+t^{-2} q^{-2s})
$$

The local $L$-factor for $L(s, \AI \sbc(\pi_1\otimes \eta^{-1}))$ is 
$(1-t^2q^{-2s})(1-t^{-2} q^{-2s}).$
Twisting by $\mu$ flips the signs, as required, because 
 $\mu$ is 
an extension of $\chi_{E/F}.$ 
Thus, $\mu$ 
maps the uniformizer $\varpi_v$ of $F_v$ to $1$ if $v$
splits and $-1$ if $v$ is inert. But for inert $v$ there is a unique completion $E_w$ of $E$ over $F_v$
and it has the same uniformizer.
\end{proof}
The corresponding formulae for the lift from $\RU_1\times \RU_1\times \RU_1$ are similar and proved in the same way.
\begin{prop}
Take $\eta_1, \eta_2,$ and  $\eta_3$ be irreducible automorphic representations of $\RU_1(\A),$ and let $\eta$ denote the representation  $\eta_1 \otimes \eta_2 \otimes \eta_3$ of $\RU_1(\A) \times \RU_1(\A) \times \RU_1(\A).$ Let $S$ be a finite set of places of $F,$ including all archimedean places and all places where  any of $\eta_1, \eta_2,$ and  $\eta_3$ is ramified. Let $T$ denote the set of places of $E$ lying above $S$ and let $\sbc$ denote the stable base change lifting from 
$U_1(\A_F)$ to  $\A_E^\times.$ Then 
\begin{equation}\label{Ad of eta111 lift}
\begin{aligned}
L^S( s, \eta, \Ad \circ \xi^{(1,1,1)}) 
&= L^S(s, \chi_{E/F})^2
L^T(s, \sbc(\frac{\eta_1}{\eta_2}))
L^T(s, \sbc(\frac{\eta_1}{\eta_3}))
L^T(s, \sbc(\frac{\eta_2}{\eta_3}))
\\
L^S( s, \eta, \Ad' \circ \xi^{(1,1,1)}) 
&=  \zeta^S(s)^2
L^T(s, \sbc(\frac{\eta_1}{\eta_2}))
L^T(s, \sbc(\frac{\eta_1}{\eta_3}))
L^T(s, \sbc(\frac{\eta_2}{\eta_3}))
\end{aligned}
\end{equation}
\end{prop}

\subsection{Conclusions}
From formulas \eqref{Ad of rog21}, \eqref{Ad of eta111 lift} and \eqref{eq relating Ad,Ad' Asai for U21} it's easy to see 
that $L^S(s, \sbc(\pi), \Asai^+)$ has a pole at $s=1$ of order equal to the number of isobaric summands in $\sbc(\pi),$
while $L^S(s, \pi, \Ad')$ has a pole of one lower order. Also, $L^S(s, \sbc(\pi), \Asai^-)$ and $L^S(s, \pi, \Ad)$ are
holomorphic and nonvanishing at $s=1.$  
Indeed, it suffices to see that the other $L$-functions appearing in  \eqref{Ad of rog21}, \eqref{Ad of eta111 lift}
are holomorphic and nonvanishing on the line $\Re(s)=1.$ For the $L$-functions $L^T(s, \sbc(\eta_i/\eta_j))$ 
this follows from the fact that $\sbc(\eta_1), \sbc(\eta_2),$ and $\sbc(\eta_3)$ are distinct. 
For $L^S(s, \mu \otimes \AI\sbc(\pi_1\otimes \eta^{-1}))$ it follows from the fact that $\AI\sbc(\pi_1\otimes \eta^{-1})$
is either an irreducible cuspidal automorphic representation of $\GL_4(\A),$ or an isobaric sum 
of two  irreducible cuspidal automorphic representations of $\GL_2(\A),$ \cite{ArthurClozel}.

\begin{thm}\label{thm: main for nonsplit case}
Take $\pi$ a globally generic irreducible cuspidal automorphic representation of $\RU_{2,1}(\A).$ Then $L^S(s, \pi, \Ad)$ is holomorphic and nonvanishing at $s=1,$
while $L^S(s, \pi, \Ad')$ can have at most a double pole. More precisely, we have three sets of equivalent conditions. 
\begin{enumerate}
\item 
\label{nonzplit main: 3}
The following are equivalent: 
\begin{enumerate}
\item $L^S(s, \pi, \Ad')$ is holomorphic and nonvanishing at $s=1$
\item $\sbc(\pi)$ is cuspidal 
\item $\pi$ is not endoscopic
\item $\pi$ is not $H'_\rho$-distinguished.
\end{enumerate}
\item 
\label{nonsplit main:(2,1)}
The following are equivalent: 
\begin{enumerate}
\item $L^S(s, \pi, \Ad')$ has a simple pole at $s=1$
\item $\sbc(\pi)$ is the isobaric sum of a character and a cuspidal representation of $\GL_2(\A_E)$
\item $\pi$ is an endoscopic lift from $\RU_{1,1}(\A) \times \RU_1(\A),$ but not from  $\RU_{1}(\A) \times \RU_{1}(\A) \times \RU_{1}(\A).$
\end{enumerate}
\item 
\label{nonsplit main:(1,1,1)}
The following are equivalent: 
\begin{enumerate}
\item $L^S(s, \pi, \Ad')$ has a double pole at $s=1$
\item $\sbc(\pi)$ is the isobaric sum of three characters of $\A_E^\times,$
\item $\pi$ is an endoscopic lift from $\RU_{1}(\A) \times \RU_{1}(\A) \times \RU_{1}(\A).$
\end{enumerate}
\end{enumerate}
\end{thm}
\subsection{Other poles}
The previous theorem gives fairly complete results for $L^S(s, \pi, \Ad' \times \chi)$ when $\chi$ is trivial or $\chi_{E/F}.$ Combining it with our earlier result, 
we have a gap: we do not know whether  $L^S(s, \pi, \Ad' \times \chi)$  can have a pole at $s=1$ when $\chi$ is cubic, or when $\chi$ is a quadratic character
other than $\chi_{E/F}.$ It turns out that the best way to proceed is by cases, based on the number of isobaric summands in 
the stable base change lift $\sbc(\pi).$ 

\begin{thm}\label{thm:other poles; type3}
Let $\pi$ be a globally generic, irreducible cuspidal automorphic representation of $\RU_{2,1}(\A)$ such that 
$\sbc(\pi)$ is cuspidal, and $\chi$ a character of $\A^\times.$ Then 
the following are equivalent
\begin{enumerate}
\item \label{ad' pole}$L^S(s, \pi, \Ad'\times \chi) $ has a pole at $s=1$
\item \label{asai+ pole}$L^S(s, \sbc(\pi) \otimes \wt \chi, \Asai^+)$ has a pole at $s=1$ for some/any
character $\widetilde \chi$ of $\A_E^\times$ whose restriction to $\A^\times$ is $\chi,$
\item \label{twist in im sbc}$ \sbc(\pi) \otimes \wt \chi$ is itself in the image of stable base change for some/any
character $\widetilde \chi$ of $\A_E^\times$ whose restriction to $\A^\times$ is $\chi.$
\end{enumerate}
\end{thm}
\begin{proof}
The meaning of \eqref{asai+ pole} is the same if ``some'' is replaced by ``any'' because 
\begin{equation}
\label{Asai+chi/chi = Ad'chi}
\frac{ L^S(s, \sbc(\pi) \otimes \wt \chi, \Asai^+)}{L^S(s, \chi)}= \frac{L^S(s, \sbc(\pi), \Asai^+ \times \chi)}{L^S(s, \chi)}
=
L^S(s, \pi, \Ad'\times \chi) ,
\end{equation}
for any $\wt \chi: \A^\times_E \to \C^\times$ with $\wt \chi\big|_{\A^\times} = \chi.$ 
The meaning of \eqref{twist in im sbc} is the same if ``some'' is replaced by ``any'' because a character
of $\A_E^\times$ which is trivial on $\A^\times$ is in the image of stable base change from $\RU_1(\A),$ 
and twisting by such a character preserves the image of stable base change from $\RU_{2,1}(\A).$ 

Equation\eqref{Asai+chi/chi = Ad'chi} also makes the first equivalence of \eqref{ad' pole} and \eqref{asai+ pole} clear. The equivalence 
of \eqref{asai+ pole} and \eqref{twist in im sbc} follows from theorem \ref{thm: image of sbc}.
\end{proof}
\begin{rmk}
{\rm If $ \sbc(\pi) \otimes \wt \chi$ is  in the image of stable base, then $ \sbc(\pi)$ and $ \sbc(\pi) \otimes \wt \chi$ both have central characters which are trivial on $\A^\times \into \A_E^\times.$ 
This implies that $\chi^3 =1.$ Thus, this case is very similar to the split case. }
\end{rmk}

\begin{thm}\label{thm:other poles: type 2+1}
Let $\pi$ be a globally generic, irreducible cuspidal automorphic representation of $\RU_{2,1}(\A)$ such that 
$\sbc(\pi)=\pi_1 \boxplus \chi_1$, for some irreducible cuspidal automorphic 
representation $\pi_1$ of $\Res_{E/F}\GL_2(\A)$ and character $\chi_1$ of $\A_E^\times$, and let $\chi$ be a nontrivial 
character of $\A^\times.$ 
So $\pi = \Rog^{(2,1)}_{w_0, \mu}(\sigma \otimes \eta_1)$, $\chi_1 = \sbc(\eta_1)$ and $\pi_1 = \mu \otimes \sbc( \sigma),$ 
for some irreducible cuspidal representation $\sigma$ of $\RU_{1,1}(\A),$ and  character $\eta_1$ of $\A_E^\times.$ Then the following are equivalent 
\begin{enumerate}
\item 
$L^S(s, \pi, \Ad' \times \chi)$ has a pole at $s=1$ 
\item 
 $L^S(s, \sigma, \Ad' \times \chi)$ has a pole at $s=1$
 \item 
 $L^S(s, \pi_1, \Asai^- \times \chi)$ (which equals  $L^S(s, \pi_1\otimes \wt \chi, \Asai^-)$  
 for any $\wt \chi: \A^\times_E \to \C^\times$ with $\wt \chi\big|_{\A^\times} = \chi$) has a pole at $s=1$
 \item 
 $\sbc(\sigma) \otimes\wt \chi$ is itself in the 
 image of $\sbc$ 
 for any $\wt \chi: \A^\times_E \to \C^\times$ with $\wt \chi\big|_{\A^\times} = \chi.$ 
\end{enumerate}
\end{thm}
\begin{proof}
If we twist all the $L$-functions in \eqref{Ad of rog21} by $\chi$ then all are holomorphic and nonvanishing except 
possibly for $L^S(s, \sigma, \Ad' \times \chi).$ This proves the first equivalence.  The second equivalence 
follows from twisting \eqref{eq relating Ad,Ad' Asai for U21}. The third follows from theorem \ref{thm: image of sbc}.
\end{proof}
\begin{rmk}
{\rm As before, if $\sbc(\sigma) \otimes\wt \chi$ is itself in the 
 image of $\sbc,$ then $\chi^2=1.$ }
\end{rmk}
To analyze the case when 
 $\sbc(\pi)$ is the isobaric sum of three characters,  i.e., that $\pi=\Rog^{(1,1,1)}(\eta_1\otimes \eta_2\otimes \eta_3),$
 we need to note that  $L^T (s, \sbc(\eta_i/\eta_j)=L^S(s, \on{AI}_{E/F}\sbc(\frac{\eta_i}{\eta_j})\times \chi).$
\begin{equation}
\label{twisted Ad' for 111}
L^S(s, \pi, \Ad' \times \chi) = L^S(s, \chi)^2
\prod_{1\le i < j \le 3}
L^S(s, \on{AI}_{E/F}\sbc(\frac{\eta_i}{\eta_j})\times \chi).
\end{equation}

By \cite{ArthurClozel}, $\on{AI}_{E/F}\sbc(\frac{\eta_i}{\eta_j})$ is a cuspidal representation of $\GL_2(\A)$ unless
$\sbc(\frac{\eta_i}{\eta_j})$ is also in the image of the lifting $\bc.$ 
\begin{lem}
If a character $\wt \chi$ of $\A_E^\times$ is in the image of stable base change from $\RU_1^{E/F}(\A),$
it is in the image of  
base change from $\A^\times,$  if and only if it is quadratic. 
\end{lem}
\begin{proof}
Write $\Fr$ for the nontrivial element of $\Gal(E/F).$ We view it as an automorphism of $\A_E^\times.$
A character is in the image of base change if and only if it satisfies 
$\wt\chi \circ \Fr =\wt \chi.$ 
(Cf. \cite[Theorem 4.2, 5.1]{ArthurClozel}.)
It is in the image of stable base change if and only if 
its restriction to $\A^\times$ is trivial. 
(This is a special case of the characterization in section \ref{ss:im sbc}: note that 
for $\Res_{E/F}\GL_1,$ we have $L^S(s, \chi, \Asai) = L^S(s, \chi|_{\A^\times}).$)
Now, given that $\wt\chi$ is trivial of $\A^\times$ we get $\wt \chi(a \Fr(a))=1$ for all $a \in \A_E^\times.$ 
Hence $\wt \chi \circ \Fr = \wt \chi^{-1}.$ Hence $\wt \chi \circ \Fr = \wt \chi \iff \wt \chi = \wt \chi^{-1}.$
\end{proof}

\begin{prop}\label{prop: other poles, type 1+1+1}
The expression \eqref{twisted Ad' for 111} 
can have at most a simple pole. It 
has a simple if $\sbc \chi_i/\chi_j = \bc \chi$ for some (necessarily unique) 
$1 \le i < j \le 3.$
\end{prop}
\begin{proof}
This follows from the previous lemma. 
We must show that the equality $\sbc \chi_i/\chi_j = \bc \chi$
can not hold for more than one pair $(i,j).$ 
Since we know that the components
$\sbc(\chi_i) , 1\le i \le j$ are distinct, it follows that 
$\sbc \chi_1/\chi_3 \ne \sbc \chi_1/\chi_2, \sbc \chi_2/\chi_3.$ 

Moreover, if $\sbc \chi_1/\chi_2 = \sbc \chi_2/\chi_3 = \xi,$ then $\xi$ must not be quadratic, 
for if it were, then $\sbc\chi_1$ would equal $\sbc\chi_3.$ But then, 
since $\xi$ is in the image of $\sbc$ and not quadratic, it is not in the image of $\bc.$
\end{proof}


\begin{thebibliography}{XXXX}
\addtocontents{Bibliography}

\bibitem[AC]{ArthurClozel}
J. Arthur and L. Clozel,  {\em Simple algebras, base change, and the advanced theory of the trace formula}, Annals of Math. Studies, no.120, Princeton University Press, (1989).



\bibitem[BZ]{BZ}
J. Bernstein, A.V. Zelevinski, {\em Representations of  the group $\GL(n,F)$, where $F$ is non-archimedian local field,} Russian Math Surveys {\bf 31:3}, (1976), 1-68.

\bibitem[BZ77]{BZ77}
J. Bernstein, A.V. Zelevinski, {\em Induced representations of reductive $\fp$-adic groups. I}, Annales Scientifiques de l.\'E.N.S, \textbf{10} (1977), 441-472.



\bibitem[Bo]{BorelCorvallis}
A. Borel, {\em Automorphic L-functions. Automorphic forms, representations and L-functions} (Proc. Sympos. Pure Math., Oregon State Univ., Corvallis, Ore., 1977), Part 2, pp. 27--61, Proc. Sympos. Pure Math., XXXIII, Amer. Math. Soc., Providence, R.I., 1979.

\bibitem[BG]{BumpGinzburg}  D. Bump, D.  Ginzburg, {\em The adjoint L-function of $GL(4).$} J. Reine Angew. Math. \textbf{505} (1998), 119–172.

\bibitem[BuZh]{BuZh}
J. Buttcane, F. Zhou, {\em Plancherel distribution of Satake parameters of Maass cusp forms on $GL_3$}, International Mathematics Research Notices, published online, https://doi.org/10.1093/imrn/rny061

\bibitem[CS]{CS}
W. Casselman, J. Shalika, {\em The unramified principal series of $p$-adic groups, II, the Whittaker functions}, Compositio Math. \textbf{41} (1980), 207-231.

\bibitem[Ch]{Chang}
B. Chang, {\em The Conjugate Classes of Chevalley Groups of Type ($G_2$),} J.  Algebra, \textbf{9}, (1968), 190-211.

\bibitem[CPS]{CPS}
J. Cogdell. I.I. Piatetski-Shapiro, {\em Derivatives and $L$-functions for $\GL_n$}, in ``Representation Theory, Number Theory and Invariant Theory", pp115-173, Progress in Mathematics, Vol 323, 2017.

\bibitem[F]{Flicker} Y. Flicker,
{\em The adjoint representation $L$-function for $GL(n)$}, Pacific J. Math. \textbf{154} (1992), no. 2, 231--244.

\bibitem[GeJ]{GelbartJacquet} S. Gelbart and H. Jacquet, 
{\em A relation between automorphic representations of $GL(2)$ and $GL(3)$.} Ann. Sci. \'Ecole Norm. Sup. (4) \textbf{11} (1978), no. 4, 471--542.

\bibitem[GeK]{GeK}
S. S. Gelbart and A. W. Knapp, {\em L-indistinguishability and R groups for the special linear group,} Adv. in Math. \textbf{43} (1982), 101-121.

\bibitem[GeRoSo93]{GeRoSo2} S. Gelbart, J. Rogawski, D. Soudry, 
{\em On periods of cusp forms and algebraic cycles for $\RU(3)$}, Israel J. Math. \textbf{83} (1993), no. 1-2, 213--252.

\bibitem[GeRoSo97]{GeRoSo1} S. Gelbart, J. Rogawski, D. Soudry, 
{\em Endoscopy, theta-liftings, and period integrals for the unitary group in three variables. }
Ann. of Math. (2) \textbf{145} (1997), no. 3, 419--476. 


\bibitem[G]{G}
D. Ginzburg, {\em A Rankin-Selberg integral for the adjoint representation of $\GL_3$}, Invent. math. \textbf{105} (1991), 571-588.

\bibitem[GH04]{GH1}
D.  Ginzburg,  J. Hundley, {\em Multivariable Rankin-Selberg integrals for orthogonal groups} Int. Math. Res. Not. 2004, no. \textbf{58}, 3097-3119.

\bibitem[GH07]{GH2}
D. Ginzburg, J. Hundley, {\em On Spin L-functions for $GSO_{10}$, }
J. Reine Angew. Math. \textbf{603} (2007), 183-213.

\bibitem[GH08]{GH3}
D. Ginzburg, J. Hundley, 
{\em The adjoint L-function for $\GL_5.$} (English summary)
Electron. Res. Announc. Math. Sci. \textbf{15} (2008), 24–32. 

\bibitem[GJ]{GJ}
D. Ginzburg, D. Jiang, {\em Siegel-Weil identity for $G_2$ and poles of $L$-functions}, Journal of Number Theory \textbf{82} no.2(2000), 256-287.


\bibitem[GRS97]{GRS-PeriodsPoles}  D. Ginzburg, S. Rallis, and D. Soudry,
{\em Periods, poles of L-functions and symplectic-orthogonal theta lifts},
J. Reine Angew. Math. \textbf{487} (1997), 85--114.

\bibitem[GRS11]{GRSBook} D. Ginzburg, S. Rallis, and D. Soudry, 
{\em The descent map from automorphic representations of $\GL(n)$ to classical groups}, World Scientific Publishing Co. Pte. Ltd., Hackensack, NJ, 2011.



\bibitem[GoJ]{GoJ}
R. Godement, H. Jacquet, {\em Zeta functions of simple algebras}, Lecture Notes in Mathematics, \textbf{260}, Springer-Verlag, 1972.


\bibitem[H12]{H12}
J. Hundley, {\em The adjoint L-function of $\RS\RU_{2,1}$}, Multiple Dirichlet series, L-functions and automorphic forms,
193-204, Progr. Math., 300, Birkhauser/Springer, New York, 2012.

\bibitem[H18]{H18}
J. Hundley, {\em Holomorphy of adjoint $L$-functions for quasi-split $A_2$}, Research in Number Theory \textbf{4}:44 (2018),

\bibitem[HS]{HS}
J. Hundley and E. Sayag, 
{\em Descent construction for GSpin groups}, Mem. Amer. Math. Soc. 243 (2016), no. 1148, 

\bibitem[J77]{Jacquet}  H. Jacquet, {\em Generic representations}, in {\it Non-commutative harmonic analysis (Actes Colloq., Marseille-Luminy, 1976)},  pp. 91--101. Lecture Notes in Math., Vol. 587, Springer, Berlin, 1977.

\bibitem[J12]{J12}
H. Jacquet, {\em A correction to Conducteur des repr\'esentations du groupe lin\'eaire}, Pacific Journal of Math, \textbf{260} (2012), 515-525.


\bibitem[JS81]{JS-EulerProductsII} H. Jacquet and J.A. Shalika, 
{\em On Euler products and the classification of automorphic representations. I.}  Amer. J. Math.  \textbf{103}  (1981), no. 3, 499--558.

\bibitem[JS88]{JS-ExteriorSquare} H. Jacquet and J.A. Shalika, Exterior square $L$
functions, Automorphic Forms, Shimura Varieties and $L$-Functions,
Vol.2 Academic Press (1990), 143-226.

\bibitem[JPSS79]{JPSS79}
H. Jacquet, I. Piatetski-Shapiro, and J. Shalika, {\em Automorphic forms on $\GL(3)$, I}, Annals of Math, second series, \textbf{109}, No. 1 (1979),  169-212.

\bibitem[JPSS81]{JPSS81}
H. Jacquet, I. Piatetski-Shapiro, and J. Shalika, {\em Conducteur des repr\'esentations du groupe lin\'eaire}, Math. Ann., \textbf{256}(2) (1981), 199-214.

\bibitem[JPSS83]{JPSS83}
H. Jacquet, I. Piatetski-Shapiro, and J. Shalika, {\em Rankin-Selberg convolutions}, American Journal of Mathematics
\textbf{ 105} No. 2 (1983), 367-464.

\bibitem[JZ]{JZ}
H. Jacquet and D. Zagier
{\em Eisenstein series and the Selberg trace formula. II.} Trans. Amer. Math. Soc. \textbf{300} (1987), no. 1, 1--48. 

\bibitem[JngR]{JngR}
D. Jiang and S. Rallis, {\em Fourier coefficients of Eisenstein series of the exceptional group of type $G_2$}, Pacific J. Math. \textbf{181} (1997), 281-314.

\bibitem[KK1]{KimKrishOdd} H. Kim and M. Krishnamurthy, {\em Base change lift for odd unitary groups}. Functional analysis VIII, 116--125, Various Publ. Ser. (Aarhus), 47, Aarhus Univ., Aarhus, 2004.

\bibitem[KK2]{KimKrishEven} H. Kim and M. Krishnamurthy,   {\em Stable base change lift from unitary groups to $\GL_n$}. IMRP Int. Math. Res. Pap. 2005, no. 1, 1--52
%
%
%
\bibitem[Ma13]{Ma13} 
N. Matringe, {\em Essential Whittaker functions for $\GL(n)$}, Doc. Math. \textbf{18} (2013), 1191-1214.

\bibitem[M13]{M13}
M. Miyauchi, {\em On local new forms for unramified $\RU(2,1)$}, Manuscripta Mathematica, \textbf{ 141}, Issue 1 (2013), 149-169.

\bibitem[M14]{M14}
M. Miyauchi, {\em Whittaker functions associated to newforms for GL(n) over p-adic fields,} Journal of the Mathematical Society of Japan, \textbf{66} (2014), 17-24.
\bibitem[Mok]{Mok}
C.P. Mok, 
{\em Endoscopic classification of representations of quasi-split unitary groups. }
Mem. Amer. Math. Soc. 235 (2015), no. 1108,

\bibitem[Ree]{Ree} R. Ree, {\em A Family of Simple Groups Associated with the Simple Lie Algebra of Type ($G_2$)}, American Journal of Mathematics, Vol. \textbf{83}, No. 3 (Jul., 1961), pp. 432-462

\bibitem[Ro]{RogawskiBook} J. Rogawski, {\em Automorphic representations of unitary groups in three variables},
Annals of Mathematics Studies, 123. Princeton University Press, Princeton, NJ, 1990.

\bibitem[Se]{Selberg} A. Selberg, 
{\em Old and new conjectures and results about a class of Dirichlet series}, Proceedings of the Amalfi Conference on Analytic Number Theory (Maiori, 1989), 367--385, Univ. Salerno, Salerno, 1992. 

\bibitem[Sh]{Sh}
F. Shahidi, {\em Functional equations satisfied by certain L-functions}, Compositio Math. \textbf{37} 2 (1978), 171-207.

\bibitem[Shim]{Shimura}
G. Shimura, 
{\em On the holomorphy of certain Dirichlet series}, Proc. London Math. Soc. (3) \textbf{31} (1975), no. 1, 79--98.

\bibitem[Shin]{Shin}
T. Shintani, {\em  On an explicit formula for class-1 ``Whittaker functions" on $\GL_n$ over $\fP$-adic fields}. Proc. Japan Acad., \textbf{52}(1976) 180–182.


 \bibitem[Sha81]{Sha81}
F. Shahidi, {\em On Certain $L$-functions}, Am. J. Math \textbf{103}, (1981), no. 2, pp. 297-355. 

 \bibitem[Sha90]{Shahidi-Plancharel}
F. Shahidi, {\em A proof of Langlands' conjecture on Plancherel measures; complementary series for $p$-adic groups},  Ann. of Math. (2)  \textbf{132}  (1990),  no. 2, 273-330.

\bibitem[So95]{Soudry-Archimedean} D. Soudry, 
{\em On the Archimedean theory of Rankin-Selberg convolutions for $SO_{2l+1} \times GL_n$},
Annales scientifiques de l' \'E.N.S. 4$^e$ s\'erie, tome 28, n$^o$ 2 (1995), p. 161-224.
 


\bibitem[Sp]{Springer-LinAlgGps}  T.A. Springer, {\sl Linear Algebraic Groups, 2nd Ed.,}
Birkh\"auser, 1998.


\bibitem[Tate1]{Tate-NumberTheoreticBackground}J. Tate,
{\em Number theoretic background}, in Automorphic forms, representations and $L$-functions (Proc. Sympos. Pure Math., Oregon State Univ., Corvallis, Ore., 1977), Part 2, pp. 3--26,
Proc. Sympos. Pure Math., XXXIII, Amer. Math. Soc., Providence, R.I., 1979.
MR0546607 

\bibitem[Ti18a]{Tianthesis}
F. Tian, {\em On the local theory of certain global zeta integrals and related problems}, Thesis, University Minnesota, 2018.

\bibitem[Ti18b]{Ti18}
F. Tian, {\em On the Archimedean Local Gamma Factors for Adjoint Representation of $\GL3$, Part I}, arXiv:1811.02752.

\bibitem[Ti18c]{Ti18c}
F. Tian, {\em On the Archimedean Local Gamma Factors for Adjoint Representation of $\GL3$, Part II}, preprint.

\bibitem[W92]{Wallach-RRG2} N. R. Wallach, {\it Real Reductive Groups II}, Volume \textrm{132} in \textit{Pure and Applied Mathematics}, Academic Press. Inc., (1992).



\end{thebibliography}
\end{document}